\documentclass[reqno,10pt]{amsart}
\usepackage{amsmath,amsthm,amssymb}
\usepackage{enumerate}
\usepackage{color}
\usepackage{mathrsfs}
\usepackage{ulem}
\usepackage{comment}
\newcommand{\I}{I_{[~\cdot\, \geq 0]}}

\newtheorem{Th}{Theorem}[section] 
\newtheorem{Lem}[Th]{Lemma} 
\newtheorem{Prop}[Th]{Proposition} 
\newtheorem{Cor}[Th]{Corollary} 
\newtheorem{Def}[Th]{Definition} 
 
\newtheorem{Rem}[Th]{Remark} 
 
\newcommand{\stasol}{z}

\def\N{{\mathbb N}} 
\def\R{{\mathbb R}}

\def\cH{{\mathcal H}}

\def\1/2{\frac{1}{2}} 
\def\dint{\int\hspace{-1ex}\int}

\def\Isom{\mbox{\rm Isom}}

\def\ov{\overline}

\def\GammaD{{\Gamma_{\rm\scriptsize D}}}
\def\GammaN{{\Gamma_{\rm\scriptsize N}}}

\renewcommand{\d}{\text{\rm d}}

\makeatletter 
\long\def\@makefntext#1{\parindent 1em\noindent 
\@hangfrom{\hbox to 1.8em{\hss$^{\@thefnmark}$}}#1}
\makeatother

\begin{document}

\title{Unidirectional evolution equations of diffusion type}

\author{Goro Akagi \quad and \quad Masato Kimura}
\address[Goro Akagi]{Graduate School of System Informatics, Kobe University,
1-1 Rokkodai-cho, Nada-ku, Kobe 657-8501 Japan}
\email{akagi@port.kobe-u.ac.jp}

\address[Masato Kimura]{Faculty of Mathematics and Physics, 
Institute of Science and Engineering, Kanazawa University, 
Kakuma, Kanazawa 920-1192 Japan 
}
\email{mkimura@se.kanazawa-u.ac.jp}

\date{\today}



\begin{abstract} 
 This paper is concerned with the uniqueness, existence, comparison
 principle and long-time behavior of solutions to the
 initial-boundary value problem for a \emph{unidirectional diffusion
 equation}. The unidirectional evolution often appears in Damage
 Mechanics due to the strong irreversibility of crack propagation or
 damage evolution. The existence of solutions is proved in an
 $L^2$-framework by introducing a
 peculiar discretization of the unidirectional diffusion equation by
 means of variational inequities of obstacle type and by developing a
 regularity theory for such variational inequalities. The novel
 discretization argument will be also applied to prove the comparison
 principle as well as to investigate the long-time behavior of solutions.
\end{abstract} 

\maketitle


\section{Introduction}\label{sec1} 

In Damage Mechanics, the \emph{unidirectional evolution} is a significant
feature of crack propagation models. Indeed, crack propagation is an
irreversible phenomena, and particularly, cracks in a specimen or the
damage of a material (e.g., microcracks which break or weaken bonds of
microstructures) cannot disappear nor decrease. Hence if one introduces a
phase parameter (see~\cite{Mi04,KMZ08,T-K09,K-T10,KRC13}) or an internal
variable (see~\cite{BaPr93,BeBi97}) which describes the crack growth
or the damage accumulation, they are usually supposed to be
unidirectional, i.e., nondecreasing or nonincreasing. Such
unidirectional evolution processes are often described by PDEs involving
the \emph{positive-part function}, $s \mapsto (s)_+ := s \vee 0 =
\max\{s,0\}$ for $s \in \R$. Moreover, in a phase field approach, the evolution of a phase
parameter is governed by the gradient flow of an appropriate energy
functional subject to some unidirectional constraint.

In order to find out mathematical features of such unidirectional
evolutions and to develop mathematical devices to analyze them, in this
paper, we shall treat, as a simplest case, the evolution of $u = u(x,t)$
governed by the (fully nonlinear) PDE,
\begin{align}\label{eqn}
\partial_t u
=\big(
\Delta u +f
\big)_+,
\quad \mbox{ for } \ x \in \Omega, \ t > 0,
\end{align}
where $\Omega$ is a bounded Lipschitz domain of $\R^n$ with $n\in\N$,
$\partial_t u = \partial u/\partial t$, $\Delta$ stands for the
$n$-dimensional Laplacian, $f = f(x,t)$ is a given function and $(s)_+ := s
\vee 0$ for $s \in \R$. More precisely, the main purpose of this paper
is to prove the uniqueness, existence and comparison principle of
\emph{strong} solutions $u = u(x,t)$ of the initial-boundary value
problem for \eqref{eqn} and to reveal the asymptotic behavior of $u =
u(x,t)$ as $t \to \infty$.

Solutions of \eqref{eqn} entail \emph{unidirectional} nature,
more precisely, the non-decrease of $u = u(x,t)$ in $t$, since the
right-hand side of \eqref{eqn} is non-negative due to the presence of the
positive part function.
There appear various kinds of unidirectional evolutions in natural
sciences and engineering fields (see, e.g., \cite{Fre01}). 
In particular, a phase field model for crack propagation in an elastic
material was proposed with the aid of a unidirectional gradient flow by
the second author and his collaborator in~\cite{K-T10, Tak09,
T-K09} (see also~\cite{AmbTor,FraMar98}). Equation \eqref{eqn} can be
regarded as a simplified equation of their model.

In mathematical points of view, \eqref{eqn} is classified as a
fully nonlinear PDE, which is not fit for energy methods in general;
however, by taking a (multi-valued) inverse function of the positive
part function $(\,\cdot\,)_+$, \eqref{eqn} can be formulated as
a sort of doubly nonlinear evolution equations,
\begin{equation}\label{irr}
\partial_t u + \alpha (\partial_t u) -\Delta u \ni f \
 \mbox{ a.e.~in } \Omega \times (0,\infty),
\end{equation}
where $\alpha$ is a (multi-valued) maximal monotone function in $\R$
given by $\alpha(0) = (-\infty,0]$ and $\alpha(s) = \{0\}$ for any $s
> 0$ with the domain $D(\alpha) = [0,\infty)$ (see Remark \ref{R:def} and
Section \ref{sec-uni} below for more details). Equation \eqref{irr} is
fitter for energy methods and monotone techniques.
On the other hand, in view of the $L^2$-theory of evolution equations, two
operators $v \mapsto \alpha (v(\cdot))$ and $v \mapsto -\Delta v$
(defined for $v \in L^2(\Omega)$) are unbounded in $L^2(\Omega)$, and
hence, it is more delicate to establish a priori estimates for proving
the existence of strong solutions, as compared with standard equations
without unidirectional constraints, e.g., the
classical and nonlinear diffusion equations.

The nonlinear PDE \eqref{irr} may fall within the frame of abstract
doubly nonlinear evolution equations in a Hilbert space $H$ of the form
\begin{equation}\label{DNE}
\partial \Psi (\partial_t u(t)) + \partial \Phi(u(t)) \ni f(t) \ \mbox{
in } H, \quad 0 < t < T,
\end{equation}
where $\partial \Psi$ and $\partial \Phi$ denote the subdifferential
operator of functionals $\Psi : H \to (-\infty,\infty]$ and $\Phi : H \to
(-\infty,\infty]$, respectively. In a thermodynamic approach to
continuum mechanics, $\Psi$ and $\Phi$ are often referred to as a
\emph{dissipation functional} and an \emph{energy functional},
respectively. To reduce \eqref{irr} into the form \eqref{DNE}, we
set $u(t) := u(\cdot,t)$ and particularly choose
$$
H = L^2(\Omega), \quad
\Psi(v) = \dfrac 1 2 \int_\Omega |v|^2 \, \d x + \I(v), \quad
\Phi(v) = \dfrac 1 2 \int_\Omega |\nabla v|^2 \, \d x
\quad \mbox{ for } \ v \in H,
$$ 
where $\I$ is the indicator function over the set $\{ v \in L^2(\Omega)
\colon v \geq 0 \ \mbox{ a.e.~in } \Omega\}$. Then we note that both
subdifferentials $\partial\Psi$ and $\partial \Phi$ are unbounded in $H$.

Let us briefly review the previous studies on abstract doubly nonlinear
evolution equations such as \eqref{DNE}.
Barbu~\cite{Barbu75} proved the existence of solutions for
\eqref{DNE} with two unbounded operators $\partial \Psi$
and $\partial \Phi$ by using the elliptic-in-time regularization and by
imposing the differentiability (in $t$) of $f$. This
result was generalized by Arai~\cite{Arai}, Senba~\cite{Senba} and
so on. In these papers, the term $\partial \Psi (\partial_t u(t))$ is
estimated by differentiating the equation and by testing it with $\partial
\Psi (\partial_t u(t))$. Therefore the differentiability of $f$ (more
precisely, $f \in W^{1,1}(0,T;H)$ by~\cite{Arai}) is essentially required, and
moreover, some strong monotonicity condition (i.e., the so-called
$\partial \Psi $-monotonicity) is also imposed on $\partial
\Phi$. Similar methods of establishing a priori estimates are also used
in individual studies on irreversible phase transition models (see,
e.g.,~\cite{BL06}). On the other hand, Colli and Visintin established an
alternative approach to \eqref{DNE} in~\cite{CV}, where $\partial \Psi$
is supposed to be bounded and coercive with linear growth instead of
assuming the regularity assumption on $f$ and the $\partial \Psi$-monotonicity
of $\partial \Phi$ (see also~\cite{Colli}). Their framework would be
more flexible in view of applications to nonlinear PDEs and has been
extensively applied to various types of doubly nonlinear
problems. Moreover, their framework has been generalized in many
directions, e.g., perturbation problems, long-time behaviors (see,
e.g.,~\cite{AiziYan, MiTh,MiRoSt08},~\cite[Sect.~11]{Roubicek},
~\cite{S08-BE, Segatti, MiRo, SSS, G11, G12, G15, G19}).
However, due to the unboundedness of $\partial \Psi$, \eqref{irr} seems
to be beyond the scope of the latter approach.
On the other hand, the former approach due to Barbu and Arai is
applicable to \eqref{irr}; however, the regularity assumption $f \in
W^{1,1}(0,T;H)$ has to be assumed for proving the existence of
solutions. Aso et al.~\cite{AsFrK, AsK} also treated an irreversible
phase transition system in a different fashion; however, the regularity
condition $f \in W^{1,2}(0,T;H)$ is also assumed there.

In this paper, we present a novel approach to \eqref{eqn} (or
equivalently \eqref{irr}) by introducing a peculiar discretization
for \eqref{eqn} by means of variational inequalities of
obstacle type. Moreover, by developing a regularity theory of such variational
inequalities, we shall establish new a priori estimates
for \eqref{eqn} (or \eqref{irr}) without assuming the differentiability
(in $t$) of $f$. Such a relaxation of the regularity assumption on $f$
may bring some advantage to consider perturbation problems, which will be left
for a forthcoming paper. Moreover, the novel discretization method will
be also applied to investigate the long-time behavior of solutions as well as
to prove a comparison theorem. In particular, we shall provide uniform
(in $t$) estimates for solutions by employing the peculiar discretization
scheme. Furthermore, some variational inequality of obstacle type will play
a crucial role in asymptotic analysis; indeed, it will turn out that
every solution will converge to the unique solution $\stasol = \stasol(x)$ of
a variational inequality of obstacle type involving the initial data as
an obstacle function from below under suitable assumptions. Here it is
worth mentioning that the limit $\stasol$ of the solution $u = u(x,t)$
depends on its initial data $u_0$; indeed, one can construct different
limits of solutions for different initial data. On the other hand, the
$\omega$-limit set of each solution must be singleton. 

The organization of this paper is as follows. 
The initial-boundary value problem for \eqref{eqn} and basic notions
(e.g., strong solution) and assumptions are formulated, and main results
of this paper will be stated in the next section. 
Section~\ref{sec-vi} is devoted to establishing a regularity theory as
well as to verifying a comparison theorem for variational inequalities of obstacle
type. In Section \ref{sec-uni}, we discuss a reduction of \eqref{eqn} to
an evolution equation of doubly nonlinear type in $L^2(\Omega)$ and
prove the uniqueness of solutions for the initial-boundary value problem.
In Section~\ref{sec-td}, we carry out the backward-Euler
time-discretization of \eqref{eqn} and construct a strong solution of
\eqref{eqn} by proposing a new a priori estimate based on the regularity
theory developed in Section~\ref{sec-vi}. A comparison theorem for
\eqref{eqn} is also proved here. The long-time behavior of solutions
will be investigated in Section~\ref{sec-ab}. In the last section, we
shall discuss other equivalent formulations of solutions for \eqref{eqn}.

\bigskip

\noindent
{\bf Notation.}
For each normed space $N$, we denote by $N'$ the dual space of $N$
with the duality pairing $\langle g,v\rangle_N := {}_{N'}\langle
g,v\rangle_N = g(v)$ for $v\in N$ and $g\in N'$.
For Banach spaces $U$ and $W$, the set of all bounded linear operators from $U$
to $W$ is denoted by $B(U,W)$. Moreover, the set of all linear topological
isomorphisms from $U$ to $W$ is denoted by $\Isom (U,W)$, that is, $A\in
\Isom (U,V)$ means that $A$ is bijective from $U$ to $W$, $A\in B(U,W)$
and $A^{-1}\in B(W,U)$. Furthermore, $\cH^k$ stands for the
$k$-dimensional Hausdorff measure in $\mathbb R^n$ for $k =
1,2,\ldots,n$. We also write $a \vee b = \max\{a,b\}$ and $a \wedge b =
\min\{a,b\}$ for $a,b \in \mathbb R$. Finally, $C$ denotes
a non-negative constant independent of the elements of the corresponding
space and set and may vary from line to line.

\section{Main results}
\label{sec-mr} 

Let $\Omega$ be a bounded Lipschitz domain in $\R^n$ with $n\in\N$.
Let $\Gamma$ be the boundary of $\Omega$ such that $\Gamma$ is the
disjoint union of two subsets $\GammaD$ and $\GammaN$, that is,
$$
\GammaD \cup \GammaN = \Gamma, \quad \GammaD \cap \GammaN = \emptyset.
$$
Moreover, assure that $\Gamma_D$ is (relatively) open in $\Gamma$. Let
$\nu$ denote the outward-pointing unit normal vector on $\Gamma$.
One of these two subsets may be empty. In such a case, the other set
coincides with the whole of $\Gamma$.
Main results of the present paper are concerned with the following
initial-boundary value problem for a unidirectional evolution equation
of diffusion type,
\begin{alignat}{4}
\partial_t u &=\big( \Delta u + f \big)_+\quad
&&\mbox{ in } \ Q := \Omega \times (0,\infty), \label{is1}\\
u&=0
&&\mbox{ on } \ \GammaD \times (0,\infty),\label{bc-D}\\
\partial_\nu u &=0
&&\mbox{ on } \ \GammaN \times (0,\infty),\label{bc-N}\\
u|_{t = 0}&=u_0
&&\mbox{ in } \ \Omega,\label{ic}
\end{alignat}
where $\partial_t = \partial/\partial t$, $f = f(x,t)$ and $u_0 = u_0(x)$ are
given functions of class $L^2_{loc}([0,\infty);L^2(\Omega))$ and
$L^2(\Omega)$, respectively, and $\partial_\nu u := \gamma_0(\nabla u)
\cdot \nu$ denotes the normal
derivative of $u$. Moreover, $(\cdot)_+$ stands for the positive part
function, i.e., $(s)_+ := s \vee 0$ for $s \in \R$. If $\GammaD$ (resp.,
$\GammaN$) is empty, the corresponding boundary condition \eqref{bc-D}
(resp., \eqref{bc-N}) is ignored.

\begin{Rem}
{\rm
By change of variable, one can reduce another unidirectional diffusion
 equation,
\begin{equation}\label{ais}
\partial_t u = \big( \Delta u + f \big)_- \quad \mbox{ in } \ Q,
\end{equation}
where $(s)_- := s \wedge 0$ for $s \in \R$, to \eqref{is1}. Indeed, set
 $v := -u$ and $g := -f$. Then \eqref{ais} is transformed to
$$
- \partial_t v = \big( - \Delta v - g \big)_-
= - \big( \Delta v + g \big)_+,
$$
whence $v$ solves \eqref{is1} with $f$ replaced by $g$.
}
\end{Rem}

Let us start with defining \emph{strong solutions} of
\eqref{is1}--\eqref{ic}. To this end, we set up notation.
Let $\gamma_0 \in B(H^1(\Omega), H^{1/2}(\Gamma))$ denote the trace
operator defined on $H^1(\Omega)$ (throughout the paper, we may omit
$\gamma_0$ if no confusion can arise). Moreover, define
\begin{align*}
V&:=\{v\in H^1(\Omega) \colon \ \gamma_0v=0
\quad \cH^{n-1}\mbox{-a.e.~on } \GammaD\},\\
X&:=\{v\in H^2(\Omega) \colon \ \gamma_0(\nabla v) \cdot \nu =0
\quad \cH^{n-1}\mbox{-a.e.~on } \GammaN\}
\end{align*}
equipped with the induced norms and inner products, i.e.,
$\|\cdot\|_V = \|\cdot\|_{H^1(\Omega)}$ and $(\cdot,\cdot)_V =
(\cdot, \cdot)_{H^1(\Omega)}$ for $V$; $\|\cdot\|_X =
\|\cdot\|_{H^2(\Omega)}$ and $(\cdot,\cdot)_X = (\cdot,
\cdot)_{H^2(\Omega)}$ for $X$.
Then $V$ and $X$ are closed subspaces of $H^1(\Omega)$
and $H^2(\Omega)$, respectively; hence, they are Hilbert spaces.
If either $\GammaD$ or $\GammaN$ is empty, the corresponding boundary
condition specified in the definition of $V$ or $X$ above is ignored.

We are concerned with \emph{strong solutions} of \eqref{is1}--\eqref{ic}
defined by

\begin{Def}[Strong solution]\label{def-ss}
For $T > 0$, a function $u \in C([0,T];L^2(\Omega))$ is called a \emph{strong
 solution} of \eqref{is1}--\eqref{ic} on $[0,T]$, if the following three
 conditions are satisfied\/{\rm :}
\begin{enumerate}[{\rm (i)}]
\item $u\in W^{1,2}(0,T;L^2(\Omega)) \cap L^2(0,T;X \cap V)$,
\item the equation $\partial_t u =(\Delta u+f)_+$ holds a.e.~in $Q_T := \Omega
	     \times (0,T)$,
\item the initial condition $u|_{t=0}=u_0$ holds a.e.~in $\Omega$.
\end{enumerate}
A function $u \in C([0,\infty);L^2(\Omega))$ is called a \emph{strong
 solution} of \eqref{is1}--\eqref{ic} on $[0,\infty)$, if for any $T> 0$, the
 restriction of $u$ onto $[0,T]$ is a strong solution
 of \eqref{is1}--\eqref{ic} on $[0,T]$.
\end{Def}

\begin{Rem}\label{R:def}
{\rm
 One can further derive that $u \in C([0,T];V)$ from (i) by employing a
 chain-rule for convex functionals. See Lemma \ref{lemforuq} in \S
 \ref{sec-uni} below for more details.
}
\end{Rem}

We are now in position to state main results, whose proofs
will be given in later sections. We begin with the uniqueness of
solutions. 

\begin{Th}[Uniqueness]\label{th-uni}
Let $T > 0$, $u_0 \in V$ and $f \in L^2(Q_T)$. Then the strong
 solution of \eqref{is1}--\eqref{ic} on $[0,T]$ is unique.
\end{Th}

To state our existence result, we shall introduce some assumptions
for the domain $\Omega$ and the boundary $\GammaD$, $\GammaN$.
For $\lambda \in \R$, we define a mapping $A_\lambda\in B(V,V')$ by
\begin{equation}\label{def:A-lam}
\langle A_\lambda u,v\rangle_V =
\int_\Omega \left(\nabla u\cdot\nabla v + \lambda uv\right)\,\d x
\quad \mbox{ for } \ u,v\in V.
\end{equation}
It is well known that $A_\lambda \in \Isom (V,V')$ holds if $\lambda >0$.
Hence one can define $u = A_\lambda^{-1} g$ for $g \in L^2(\Omega)$ as
the unique solution $u$ of the elliptic problem in a weak form,
$$
\int_\Omega \left(\nabla u\cdot\nabla v + \lambda uv\right)\,\d x
= \int_\Omega g v \, \d x
\quad \mbox{ for all } \ v \in V,
$$
(i.e., $A_\lambda u = g$ in $V'$). Then we assume that
\begin{align}\label{A1condition}
A_1^{-1}g\in H^2(\Omega)\quad \mbox{for all } \ g\in L^2(\Omega).
\end{align}
Condition \eqref{A1condition} is often called an \emph{elliptic
regularity} condition and deeply related to the geometry of the domain and
boundary conditions. Indeed, it holds true for smooth domains with a
single boundary condition (i.e., $\GammaN = \emptyset$ or $\GammaD =
\emptyset$). However, it is more delicate to consider the validity of
\eqref{A1condition} for situations with nonsmooth domains or mixed
boundary conditions. On the other hand, in order to take account of
physical backgrounds of crack growth models and their numerical
simulations, the regularity of the boundary may be at most Lipschitz
continuous, and mixed boundary conditions seem to be natural as well.
We shall give conditions equivalent to \eqref{A1condition} in
Proposition~\ref{prop-GammaD} below.

\noindent
\begin{Rem}
{\rm
Let us exhibit a couple of examples of $\Omega$, $\GammaD$ and $\GammaN$
for which the condition \eqref{A1condition} is satisfied.
\begin{enumerate}[{\rm (i)}]
\item If $\ov{\GammaD}\cap \ov{\GammaN}=\emptyset$ and $\Gamma$ is
of class $C^{1,1}$, then \eqref{A1condition} is
satisfied (see Theorem~2.2.2.3 and Theorem~2.2.2.5 of
      \cite{Gri85}).
\item
Let $\Omega$ be convex. 
If $\GammaD=\Gamma$ or $\GammaD=\emptyset$, then \eqref{A1condition} is
satisfied (see Theorem~3.2.1.2 and Theorem~3.2.1.3 of \cite{Gri85}).
\item
If $n=1$ or if $\Omega$ is a rectangle in $\R^2$ and $\GammaD$ is 
a union of some of four edges of $\Gamma$, then \eqref{A1condition} is
satisfied. Indeed, since the weak solution $u = A_1^{-1}g$ can be
     extended to an open neighborhood of $\ov{\Omega}$ by reflection, it
     follows that $u = A_1^{-1}g\in H^2(\Omega)$.
\end{enumerate}
}
\end{Rem}

Our existence result reads,

\begin{Th}[Existence]\label{T:ex}
We suppose that the condition \eqref{A1condition} holds true. 
Let $T>0$, $u_0\in X\cap V$ and $f\in L^2(Q_T)$ be given.
If there exists $f^*\in L^2(\Omega)$ satisfying 
\begin{equation}\label{f-hypo1}
f(x,t)\le f^*(x)\quad \mbox{a.e.~in }Q_T,
\end{equation}
then there exists a strong solution $u = u(x,t)$ to the
 problem \eqref{is1}--\eqref{ic} on $[0,T]$.
\end{Th}

\begin{Rem}[Assumptions on $f$]\label{R:f-hypo1}
{\rm
Condition \eqref{f-hypo1} is weaker than $f \in L^\infty(Q_T)$ or $f \in
 W^{1,1}(0,T;L^2(\Omega))$ (cf.~\cite{Arai}). Indeed, if $f \in W^{1,1}(0,T;L^2(\Omega))$,
 then $f^*(x) := f(x,0) + \int^T_0 |\partial_t f(x,t)| \, \d t$ belongs to
 $L^2(\Omega)$ and satisfies \eqref{f-hypo1}.
On the other hand, \eqref{f-hypo1} is stronger than $f_+ := f \vee 0 \in
 L^\infty(0,T;L^2(\Omega))$. In fact, \eqref{f-hypo1} yields
 $f_+ \in L^\infty(0,T;L^2(\Omega))$. However, even if $f_+
 \in L^\infty(0,T;L^2(\Omega))$, \eqref{f-hypo1} might not
 hold true. One may easily find a counterexample, e.g., $f(x,t) =
 |x-t|^{-\alpha}$, $\Omega = (0,1)$, $T = 1$ and $0 < \alpha < 1/2$.
}
\end{Rem}

Theorem \ref{T:ex} will be proved in Section \ref{sec-td}. Our method of
proof relies on the backward-Euler time-discretization of \eqref{is1}.
The discretized equation will be rewritten as some variational
inequalities of obstacle type. Moreover,
developing a regularity theory for such variational inequalities, we
shall obtain a new a priori estimate for discretized solutions. As a
result, we shall construct a strong solution without assuming the
differentiability of $f$ in $t$, which is usually assumed in standard
frameworks of doubly nonlinear evolution equations such as \eqref{dnp}
(see, e.g.,~\cite{Barbu75,Arai,Senba}).

Moreover, such a basic strategy of proving the existence result will be
also applied to prove the following \emph{comparison theorem} for strong
solutions of \eqref{is1}--\eqref{ic}:

\begin{Th}[Comparison principle]\label{T:comp}
Let $T > 0$ and suppose that \eqref{A1condition} is satisfied.
For each $i = 1,2$, let $u_0^i\in X\cap V$ and $f^i\in L^2(Q_T)$ be such
 that there exists $f^*\in L^2(\Omega)$ satisfying
\begin{align*}
f^i(x,t)\le f^*(x)\quad \mbox{a.e.~in } \, Q_T.
\end{align*}
For $i=1,2$, let $u^i=u^i(x,t)$ be the unique strong solution of
 \eqref{is1}--\eqref{ic} with $u_0=u_0^i$ and $f=f^i$ on $[0,T]$. 
If $u_0^1\le u_0^2$ a.e.~in $\Omega$ and $f^1\le f^2$ a.e.~in $Q_T$, 
then $u^1\le u^2$ a.e.~in $Q_T$.
\end{Th}

The peculiar discretization argument will be also employed to
investigate the long-time behavior of solutions for
\eqref{is1}--\eqref{ic}. Furthermore, the comparison theorem stated above
will play a crucial role to identify the limit of each solution $u =
u(x,t)$ as $t \to \infty$.

\begin{Th}[Convergence of solutions as $t \to \infty$]\label{T:asp}
 Let $u_0 \in X \cap V$ and assume that \eqref{A1condition} holds and that
\begin{enumerate}[{\rm (H1)}]
 \item
      $\cH^{n-1}(\GammaD) > 0$\,\/{\rm ;}
 \item
      there exists a function $f_\infty \in L^2(\Omega)$ such
      that $f - f_\infty$ belongs to $L^2(0,\infty;L^2(\Omega))$\/{\rm ;}
 \item 
      $f \in L^\infty(0,\infty ; L^2(\Omega))$, and \eqref{f-hypo1} is
      satisfied. 
\end{enumerate}
 Then the unique solution $u = u(x,t)$ of \eqref{is1}--\eqref{ic} on
 $[0,\infty)$ converges to a function $\stasol =
 \stasol(x) \in X \cap V$ strongly in $V$ as $t \to \infty$. Moreover, the
 limit $\stasol$ satisfies
$$
 \stasol \geq u_0 \quad \mbox{ and } \quad - \Delta \stasol \geq f_\infty
 \ \mbox{ a.e.~in } \Omega.
$$
 
 In addition, if $f(x,t) \leq f_\infty(x)$ for a.e.~$(x,t) \in Q$, then
 the limit $\stasol$ coincides with the unique solution $\bar \stasol \in X
 \cap V$ of the following variational
 inequality\/{\rm :}
$$
\mbox{\rm (VI)($u_0, f_\infty$)}
\begin{cases}
\ \displaystyle \bar \stasol \in K_0(u_0) := \{ v \in V \colon v \geq u_0
 \ \mbox{ a.e.~in } \Omega\},\\[1mm]
\ \displaystyle \int_\Omega \nabla \bar \stasol \cdot \nabla (v -
 \bar \stasol) \, \d x
\geq \int_\Omega f_\infty (v - \bar \stasol) \, \d x \quad \mbox{ for all
 } \ v \in K_0(u_0).
\end{cases}
$$
\end{Th}

\begin{Rem}
{\rm
Assumption (H1) is essentially required to ensure the convergence of
the solution $u = u(x,t)$ as $t \to \infty$. Indeed, suppose that
 $\GammaD = \emptyset$ (i.e., $\GammaN = \Gamma$) and set $u_0(x) \equiv
 1$ and $f(x,t) \equiv 1$. The unique strong solution of
 \eqref{is1}--\eqref{ic} is given by
$$
u(x,t) = 1 + t \quad \mbox{ for } \ (x,t) \in Q,
$$
and then, $u(x,t)$ is divergent to $\infty$ at each $x \in \Omega$ as $t
 \to \infty$.
}
\end{Rem}

\section{Regularity theory for variational inequalities of obstacle
 type}\label{sec-vi}

In this section, based on the approach of Gustafsson
\cite{Gus86}, we revisit a regularity theory for variational
inequalities of obstacle type. In classical literature on variational inequalities of
obstacle type, the $W^{2,p}(\Omega)$-regularity of solutions is often obtained
by using a penalization technique (see, e.g.,~\cite[Chap.~4]{K-S80},~
\cite{Fri88}). On the other hand, Gustafsson \cite{Gus86} gave a simpler
alternative proof by introducing an auxiliary variational inequality and
by proving the coincidence of solutions for both problems. Let us also
 remark that, in previous results, it is assumed that
$W^{2,p}(\Omega)\subset C(\ov{\Omega})$ (namely, $p>n/2$) in order to
utilize the classical maximum principle for linear elliptic equations.

We shall establish a $W^{2,p}(\Omega)$-regularity
result for variational inequalities of obstacle type equipped with a mixed
boundary condition by properly modifying the argument of
Gustafsson~\cite{Gus86}.
It is noteworthy that we do not assume that $W^{2,p}(\Omega) \subset
C(\ov{\Omega})$, i.e., $p > n/2$, as we employ Stampacchia's truncation
technique instead of the classical maximum principle.

First, we shall set up notation.
For $\sigma \geq 0$, let $A:=A_\sigma \in B(V,V')$ be defined as in
\eqref{def:A-lam}. Define a symmetric bilinear form
$a(\cdot,\cdot) : V\times V \to \mathbb R$ associated with $A$ by
\begin{align}\label{a1}
a(u,v):=
\langle A u, v\rangle_V =
\int_\Omega \left( \nabla u \cdot \nabla v + \sigma uv \right)\,\d x
\quad
\mbox{ for } \ u,v\in V.
\end{align}
Throughout this section, we assume that
\begin{equation}\label{A1}
\sigma >0\quad \mbox{if}\quad \cH^{n-1}(\GammaD)=0.
\end{equation}
Under the condition \eqref{A1}, (by the Poincar\'{e} inequality for the
case that $\sigma = 0$ and $\cH^{n-1}(\GammaD) > 0$), $a(\cdot,\cdot)$
turns out to be coercive on $V\times V$. Hence $A$ is invertible, and
$A^{-1}$ belongs to $B(V',V)$. Let $f\in V'$ and $\psi\in V$ and define
closed convex subsets $K_0$, $K_1$ of $V$ by
\begin{align}
K_0&:=\{v\in V \colon v\ge \psi\mbox{ a.e.~in }\Omega\},\label{K0}\\
K_1&:=\{v\in V \colon Av\ge f \mbox{ in }V'\}.\nonumber
\end{align}
Here the inequality $Av\ge f$ in $V'$ means that $\langle Av-f, \varphi
\rangle_V \geq 0$ for all $\varphi \in V$ satisfying $\varphi \geq 0$
a.e.~in $\Omega$. 
We also define functionals $J,\hat J$ on $V$ by 
\begin{align}
J(v)&:=\dfrac 1 2 a(v,v) -\langle f,v\rangle_V
\quad  \mbox{ for } \ v\in V, \label{a2}\\
\hat{J}(v)&:= \dfrac 1 2 a(v,v) -\langle \hat{f},v\rangle_V
\quad \mbox{ for } \ v\in V, \nonumber
\end{align}
where $\hat{f}:=A\psi\in V'$. Then the following equivalence holds true:

\begin{Prop}\label{th1}
Suppose \eqref{A1} and let $f\in V'$ and $\psi\in V$.
Then the following five conditions for $u \in V$ are equivalent
to each other\/{\rm :}
\begin{enumerate}[{\rm (a)}]
\item
$u\in K_0$, $J(u)\le J(v)$ for all $v\in K_0$,
\item
$u\in K_0$, $a(u,v-u)\ge \langle f,v-u\rangle_V$ for all $v\in K_0$,
\item
$u\in K_0\cap K_1$, $\langle Au-f,u-\psi\rangle_V =0$,
\item
$u\in K_1$, $a(u,v-u)\ge \langle \hat{f},v-u\rangle_V$ for all $v\in K_1$,
\item
$u\in K_1$, $\hat{J}(u)\le \hat{J}(v)$ for all $v\in K_1$.
\end{enumerate}
Moreover, there exists a unique element $u \in V$ satisfying all the
 conditions. 
\end{Prop}

\begin{proof}
Since $J$ is a coercive, continuous, strictly convex
functional on the closed convex set $K_0$, $J$ admits a unique
minimizer $u$ over $K_0$. Hence $u$ satisfies (a). 

We shall prove the equivalence of the conditions (a)--(e).
It is well known (see, e.g.,~\cite{K-S80}) that (a) $\Leftrightarrow$
 (b) and (d) $\Leftrightarrow$ (e).
So, let us here start with showing that (b) $\Rightarrow$ (c). 
The condition (b) is equivalently rewritten by
\begin{align}\label{alt}
u\in K_0,
\quad
\langle Au-f,v-u\rangle_V \ge 0
\quad
\mbox{for all }v\in K_0.
\end{align}
For any $\varphi\in V$ with $\varphi\ge 0$ a.e.~in $\Omega$, 
substituting $v=u+\varphi\in K_0$ to \eqref{alt},
we have
$\langle Au-f,\varphi\rangle_V \ge 0$,
which yields that $u\in K_1$. 
On the other hand, substitute $v=\psi\in K_0$ and $v=2u-\psi\in K_0$ to
\eqref{alt}. Then one can obtain $\langle Au-f,u-\psi\rangle_V
=0$. Hence (c) holds.

To prove the inverse relation, (c) $\Rightarrow$ (b), 
let $u$ satisfy (c). For any $v\in K_0$, we see that
\begin{align*}
a(u,v-u)-\langle f,v-u\rangle_V
=
\langle Au-f,v-\psi\rangle_V
-\langle Au-f,u-\psi \rangle_V.
\end{align*}
Since the first term of the right-hand side is nonnegative (by $u \in
K_1$ and $v \in K_0$) and the second term vanishes (by the equation of (c)), the
condition (b) follows.

One may prove the equivalence between (c) and (d) in a similar fashion to
the above. Firstly, suppose that $u$ satisfies (c). For any
$v\in K_1$, one finds that
\begin{align*}
a(u,v-u)- \langle \hat{f},v-u\rangle_V
&= \langle Av - Au, u \rangle_V - \langle Av - Au, \psi \rangle_V\\
&= \langle Av-f,u-\psi\rangle_V
-\langle Au-f,u-\psi \rangle_V.
\end{align*}
Here we used the fact that $\langle Aw,z \rangle_V = \langle Az,w
 \rangle_V$ for all $w,z \in V$.
Noting that the right-hand side is non-negative by (c) and the fact that
 $u \in K_0$ and $v \in K_1$, one can get (d).
To check the inverse relation, we also rewrite (d) as
\begin{align}\label{alt2}
u\in K_1,
\quad
\langle Av-Au,u-\psi\rangle_V \ge 0
\quad
\mbox{for all }v\in K_1.
\end{align}
For any $\varphi\in L^2(\Omega)$ with $\varphi\ge 0$ a.e.~in $\Omega$, 
substituting $v=u+A^{-1}\varphi\in K_1$ to \eqref{alt2},
we have $(\varphi,u-\psi)_{L^2(\Omega)} \ge 0$, which along with the
arbitrariness of $\varphi \geq 0$ implies $u\in
K_0$. Moreover, let us also substitute $v=A^{-1}f\in K_1$ and
$v=2u-A^{-1}f\in K_1$ in \eqref{alt2}. Then we obtain
$\langle Au-f,u-\psi\rangle_V =0$, whence (c) follows.
%
Consequently, all the conditions (a)--(e) are equivalent.
\end{proof}

In the rest of this section, let $p\in\R$ satisfy
\begin{align}\label{c1}
1<p<\infty,
\quad
p\ge \frac{2n}{n+2}.
\end{align}
Since the H\"{o}lder conjugate $q:=p/(p-1)$ of $p$ satisfies $q\le
2^*:=2n/(n-2)$ if $n\ge 3$, and $\Omega$ is a Lipschitz domain, by Sobolev's
embedding theorem, the continuous embeddings $V\hookrightarrow
L^q(\Omega)$ and $L^p(\Omega) \cong (L^q(\Omega))' \hookrightarrow V'$
hold true. We also note that $W^{2,p}(\Omega)$ is continuously embedded in
$H^1(\Omega)$ by \eqref{c1}.

We further suppose that
\begin{align}\label{c2}
f\in L^p(\Omega),
\quad
\psi \in V, 
\quad
A\psi \in L^p(\Omega)
\end{align}
and introduce a closed (in $V$) convex set $K_2$ given by 
\begin{align*}
K_2:=\{v\in V \colon f\le Av\le f \vee \hat{f} \mbox{ in }V'
 \}\subset K_1.
\end{align*}

A main result of this section is stated as follows:

\begin{Th}\label{th2}
Suppose that \eqref{A1}, \eqref{c1} and \eqref{c2} are satisfied. Then each of
 the following conditions {\rm (f)--(h)} is equivalent to the conditions
 {\rm (a)--(e)} of Proposition~\ref{th1}\/{\rm :} 
\begin{enumerate}
\item[{\rm (f)}]
$u\in K_2$, $\hat{J}(u)\le \hat{J}(v)$ for all $v\in K_2$,
\item[{\rm (g)}]
$u\in K_2$, $a(u,v-u)\ge \langle \hat{f},v-u\rangle_V$ for all $v\in
	     K_2$,
\item[{\rm (h)}]
$u\in K_0\cap K_2$, $(Au-f)(u-\psi)=0$ a.e.~in
	     $\Omega$.
\end{enumerate}
\end{Th}

We first prepare a couple of lemmas, which will be also used
in later sections.
\begin{Lem}\label{lem-trace}
Let $w\in H^1(\Omega)$ and set $w_+(x):=(w(x))_+$ for $x \in
 \Omega$. Then $w_+\in H^1(\Omega)$ and $\gamma_0(w_+)=(\gamma_0w)_+$
 $\cH^{n-1}$-a.e.~on $\Gamma$.
\end{Lem}

\begin{proof}
Set $\Omega_+:=\{x\in\Omega \colon w(x)>0\}$ and recall Theorem~A.1.\,of
 \cite{K-S80} to observe that $w_+\in H^1(\Omega)$ and
\begin{align}\label{th-KS}
w_+=\begin{cases}
     w&\mbox{a.e.~in }\Omega_+,\\
     0&\mbox{a.e.~in }\Omega\setminus \Omega_+,
    \end{cases}
\quad
\nabla w_+=\begin{cases}
	    \nabla w&\mbox{a.e.~in }\Omega_+,\\
	    0&\mbox{a.e.~in }\Omega\setminus \Omega_+.
	   \end{cases}
\end{align}
Set $W := C(\ov{\Omega})\cap H^1(\Omega)$.
Since $W$ is dense in $H^1(\Omega)$, there exists a sequence
$\{w_n\}$ in $W$ such that $w_n\to w$ strongly in $H^1(\Omega)$
as $n\to\infty$. Noting that 
\begin{align}\label{l2}
\|(w_n)_+-w_+\|_{L^2(\Omega)}\le \|w_n-w\|_{L^2(\Omega)}
\to 0
\quad
\mbox{as}
\quad
n\to \infty,
\end{align}
we observe that $(w_n)_+\to w_+$ strongly in $L^2(\Omega)$.
Applying \eqref{th-KS} to $w_n$, we also have
\begin{align*}
\|(w_n)_+\|_{H^1(\Omega)}^2
=
\|(w_n)_+\|_{L^2(\Omega)}^2
+\|\nabla (w_n)_+\|_{L^2(\Omega)}^2
\le
\|w_n\|_{L^2(\Omega)}^2
+\|\nabla w_n\|_{L^2(\Omega)}^2
=
\|w_n\|_{H^1(\Omega)}^2.
\end{align*}
Since $\{w_n\}$ is bounded in $H^1(\Omega)$, so is $\{(w_n)_+\}$.
Hence, one can extract a (non-relabeled) subsequence of $\{n\}$ such
 that $(w_n)_+ \to w_+$ weakly in $H^1(\Omega)$
Again from \eqref{th-KS}, we have
\begin{align*}
\|(w_n)_+\|_{H^1(\Omega)}^2
&=
\|(w_n)_+\|_{L^2(\Omega)}^2
+(\nabla (w_n)_+,~\nabla w_n)_{L^2(\Omega)}\\
&\to
\|w_+\|_{L^2(\Omega)}^2
+(\nabla w_+,~\nabla w)_{L^2(\Omega)}
=
\|w_+\|_{H^1(\Omega)}^2,
\end{align*}
which together with the uniform convexity of $H^1(\Omega)$ also implies
 that $(w_n)_+\to w_+$ strongly in $H^1(\Omega)$.
Since $\gamma_0\in B(H^1(\Omega),H^{1/2}(\Gamma))$,
we particularly deduce that $\gamma_0w_n\to \gamma_0w$ 
and $\gamma_0(w_n)_+\to \gamma_0w_+$ strongly in $L^2(\Gamma)$.
As in \eqref{l2}, one can verify that $(\gamma_0w_n)_+\to (\gamma_0w)_+$
 strongly in $L^2(\Gamma)$.
On the other hand, since $w_n\in C(\ov{\Omega})$, it is clear that
\begin{align}\label{trace}
\gamma_0(w_n)_+=(\gamma_0w_n)_+
\quad \mbox{ $\cH^{n-1}$-a.e.~on } \Gamma.
\end{align}
Passing to the limit as $n\to\infty$ in \eqref{trace}, we conclude that
$\gamma_0w_+=(\gamma_0w)_+$ \ $\cH^{n-1}$-a.e.~on $\Gamma$.
\end{proof}

\begin{Lem}\label{lem-min}
If $v_1, v_2\in V$, then $v_1 \vee v_2 \in V$ and $v_1 \wedge v_2 \in V$.
\end{Lem}

\begin{proof}
Applying Lemma~\ref{lem-trace} to $w = \pm (v_1 - v_2)$,
we find that $(v_1-v_2)_+$ and $(v_2-v_1)_+$ belong to $V$.
Hence we obtain $v_1 \vee v_2 \in V$ and $v_1 \wedge v_2 \in V$,
since $v_1 \vee v_2 = v_2+(v_1-v_2)_+$ and $v_1 \wedge v_2 =v_2-(v_2-v_1)_+$.
\end{proof}

Let us move on to a proof of Theorem \ref{th2}.

\begin{proof}[Proof of Theorem \ref{th2}]
It is obvious that (f) $\Leftrightarrow$ (g).
As in Proposition \ref{th1}, one can uniquely choose $u \in K_2$ which satisfies (f)
 and (g). Next, we shall prove the equivalence between (a)--(e) and (f),
 (g). Let $u_1$ be the unique element of $V$ satisfying (a)--(e) and let
 $u_2$ be the unique element of $V$ satisfying (f) and (g). 

We claim that $u_1=u_2$. 
Indeed, note that $Au_2\in L^p(\Omega)$ by $u_2\in K_2$.
Set $w:=u_2-\psi\in V$ and $h:=Aw\in L^p(\Omega)$.
Since $u_2$ satisfies (g), it follows that 
\begin{align}\label{du-2}
0 \leq a(u_2, v - u_2) - \langle \hat f , v - u_2 \rangle_V
&= \langle Au_2 - A\psi, v-u_2 \rangle_V \nonumber \\
&= \langle Av-Au_2,w\rangle_V
\quad
\mbox{for all }\ v\in K_2.
\end{align}
We set a measurable set
\begin{align*}
N:=\{x\in \Omega \colon ~w(x)<0\}.
\end{align*}
Define $g\in L^p(\Omega)$ by
\begin{align*}
g(x):=\left\{
\begin{array}{ll}
f(x) \vee \hat{f}(x) &\mbox{ if } \ x\in N,\\
Au_2(x)&\mbox{ if } \ x\in \Omega\setminus N.
\end{array}
\right.
\end{align*}
Then by definition one can observe that $A^{-1}g\in K_2$.
Hence substituting $v=A^{-1}g$ to \eqref{du-2}, we have
$$
0 \leq \langle g - Au_2, w \rangle_V = \int_N \left((f \vee \hat
 f)-Au_2\right)w\,\d x.
$$
Since $(f \vee \hat f) -Au_2\ge 0$ and $w<0$ a.e.~in $N$, one can derive
 the relation $Au_2= f \vee \hat f$ a.e.~in $N$. It follows that
\begin{align*}
h=Au_2-A\psi= (f \vee \hat f) - \hat f = (f-\hat{f})_+\ge 0
\quad
\mbox{ a.e.~in }N.
\end{align*}

On the other hand, we recall that $w\in V$ solves the equation $Aw=h$ in
 $V'$, that is, 
\begin{align}\label{wfw}
a(w,v)=\langle h,v\rangle_V
\quad
\mbox{for all }v\in V.
\end{align}
We define $w_-:= w \wedge 0$. Then by Lemma~\ref{lem-min} and
 Theorem~A.1.\,of \cite{K-S80}, we have
\begin{align*}
w_-\in V,
\quad
w_-= \begin{cases}
      w&\mbox{ in }N,\\
      0&\mbox{ in }\Omega\setminus N,
     \end{cases}
\quad
\nabla w_-= \begin{cases}
	     \nabla w&\mbox{a.e.~in }N,\\
	     0&\mbox{a.e.~in }\Omega\setminus N.
	    \end{cases}
\end{align*}
Substituting $v=w_-$ into \eqref{wfw}, we have
\begin{equation}\label{intN}
0 \leq a(w_-,w_-)=a(w,w_-)\stackrel{\eqref{wfw}}{=} \langle h,w_-\rangle_V
= \int_N hw\,\d x.
\end{equation}
Since $h\ge 0$ and $w<0$ a.e.~in $N$, the right-hand side of
 \eqref{intN} is non-positive. Hence we deduce that
$a(w_-,w_-)=0$, which along with the coercivity of $a(\cdot,\cdot)$ implies that
 $w_-=0$ (i.e., $w \geq 0$) a.e.~in $\Omega$. Therefore $u_2$ belongs to
 $K_0$.

Substitute $v=A^{-1}f\in K_2$ to the condition \eqref{du-2}. Then we obtain
\begin{align}\label{ineq}
\quad\langle Au_2-f,u_2-\psi\rangle_V \le 0.
\end{align}
Moreover, noting that $Au_2-f\ge 0$ (by $u_2 \in K_2$) and $u_2-\psi\ge
 0$ (by $u_2 \in K_0$), we derive
 $\langle Au_2-f,u_2-\psi\rangle_V =0$ by \eqref{ineq}.
Hence, $u=u_2$ satisfies the condition (c).
By uniqueness, we obtain $u_1=u_2$.
Thus we have proved that all the conditions (a)--(g) are
 equivalent. 

Finally, we note that (h) immediately implies (c), since $K_2 \subset
 K_1$. Conversely, let $u$ satisfy (c).
 Then $u$ belongs to $K_2$ by (f), and hence, $Au \in L^p(\Omega)$ and
 $$
 0 = \langle Au - f , u - \psi \rangle_V = \int_\Omega (Au - f)(u-\psi)
 \, \d x.
 $$ 
 Thus we obtain $(Au - f)(u - \psi) = 0$ a.e.~in $\Omega$, since $Au
 \geq f$ and $u \geq \psi$ a.e.~in $\Omega$. Therefore (h) holds.
\end{proof}

Thanks to Theorem \ref{th2}, for each solution $u$ of the variational
inequality of obstacle type,
\begin{equation}\label{vi}
u \in K_0, \quad a(u, v - u) \geq \langle f, v - u \rangle_V \quad 
\mbox{for all } \ v \in K_0,
\end{equation}
we have obtained an additional information, $u \in K_2$. In order to
more explicitly clarify the feature of the additional information,
let us introduce the following assumption: 
\begin{align}\label{c3}
\begin{cases}
 A_1^{-1}g\in W^{2,p}(\Omega)\quad \mbox{for all } \ g\in L^p(\Omega) 
 &\mbox{ in case } \ \sigma > 0 \ \mbox{ or } \ p \leq 2^*,\\
 A_1^{-1}g\in W^{2,\rho}(\Omega)\quad \mbox{for all } \ g\in L^\rho(\Omega)
 \ \mbox{ and } \ \rho \in [2^*,p]
 &\mbox{ in case } \ \sigma = 0 \ \mbox{ and } \ p > 2^*,
\end{cases}
\end{align}
where $2^* = 2n/(n-2)_+$ (here we note that
\eqref{A1condition} is a special case of \eqref{c3} with $p = 2$).
Then we find that
\begin{Prop}\label{P:K2}
Assume that \eqref{c3} holds. Then
$$
K_2 \subset W^{2,p}(\Omega).
$$
\end{Prop}

\begin{proof}
Let $v \in K_2$.
In case $\sigma > 0$ or $p \leq 2^*$ (hence $V \hookrightarrow
 L^p(\Omega)$), since $Av \in L^p(\Omega)$ and $v \in V$, one can check
 that $v \in L^p(\Omega)$. Hence $A_1 v = A v + (1 - \sigma) v \in
 L^p(\Omega)$, which along with \eqref{c3} gives $v \in W^{2,p}(\Omega)$.
In case $\sigma = 0$ and $p > 2^*$, noting that $A_1 v \in
 L^{2^*}(\Omega)$ as above, we find by \eqref{c3} that $v \in
 W^{2,2^*}(\Omega)$. By iteration argument along with Sobolev's
 embeddings, one can finally conclude that $v \in
 L^p(\Omega)$. Therefore we deduce that $v \in W^{2,p}(\Omega)$ as in
 the former case.
\end{proof}

Condition \eqref{c3} can be regarded as \emph{elliptic regularity} of
 weak solutions for the elliptic boundary value problem,
$$
-\Delta u + u=f \ \mbox{ in } \Omega, \quad
u=0 \ \mbox{ on } \GammaD,\quad
\partial_\nu u = 0 \ \mbox{ on } \GammaN,
$$
and it holds true in many cases, e.g., smooth domains with $\GammaN =
 \emptyset$ or $\GammaD = \emptyset$ (see, e.g.,~\cite{G-T83}). However,
 the validity of \eqref{c3} is more delicate, if $\Omega$ is not smooth
 or mixed boundary conditions are imposed. So we here explicitly made the
 assumption.
 
To take account of boundary conditions, we further define a subspace of
$W^{2,p}(\Omega)$ by 
\begin{align*}
X^p&:=\{v\in W^{2,p}(\Omega) \colon \ \gamma_0(\nabla v) \cdot \nu =0 \ \
 \cH^{n-1}\mbox{-a.e.~on } \GammaN\}.
\end{align*}
Prior to stating a regularity result for \eqref{vi}, we prepare a
proposition, which provides equivalent forms of the assumption \eqref{c3}.

\begin{Prop}\label{prop-GammaD}
Under the assumption \eqref{c1}, the operator $A_\lambda|_{X^p \cap V}$
 restricted onto $X^p \cap V$ is injective and bounded linear from
 $X^p \cap V$ into $L^p(\Omega)$, and it coincides with the operator $-\Delta +
 \lambda$, where $\Delta$ means the Laplace operator from $D(\Delta) =
 X^p \cap V$ into $L^p(\Omega)$, that is, the Laplacian equipped with
 the Dirichlet and Neumann boundary conditions on $\GammaD$ and
 $\GammaN$, respectively, in a strong form.

Moreover, the following conditions for $\Omega$ and $\GammaD$, $\GammaN$ 
are equivalent to each other\/{\rm :}
\begin{enumerate}[{\rm (i)}]
\item
there exists 
$\lambda > 0$ such that 
$A_\lambda ^{-1}g\in W^{2,p}(\Omega)$  for all $g\in
     L^p(\Omega)$\/{\rm ;}
\item
for any $\lambda > 0$, it holds that $A_\lambda ^{-1}g\in W^{2,p}(\Omega)$
for all $g\in L^p(\Omega)$\/{\rm ;}
\item
there exists $\lambda > 0$ such that $A_\lambda^{-1} g \in X^p$ for all
	     $g \in L^p(\Omega)$, and $(-\Delta +\lambda) \in \Isom
	     (X^p\cap V,L^p(\Omega))$\/{\rm ;}
\item
for any $\lambda > 0$, it holds that $A_\lambda^{-1} g \in X^p$ for all
	     $g \in L^p(\Omega)$, and
$(-\Delta +\lambda) \in \Isom (X^p\cap V,L^p(\Omega))$.
\end{enumerate}
\end{Prop}

\begin{proof}
Denote $B_\lambda :=A_\lambda |_{X^p\cap V}$ for
$\lambda >0$. Then, for $u\in X^p\cap
V$, we observe by Green's formula, which is valid for Lipschitz domains,
 that
\begin{align*}
\langle B_\lambda u,v\rangle_V =
\int_\Omega \left( \nabla u\cdot\nabla v + \lambda uv \right)\,\d x
= \int_\Omega \left( -\Delta u + \lambda u \right) v \, \d x
\quad \mbox{ for all } \ v\in V \cap W^{1,q}(\Omega),
\end{align*}
which implies that $B_\lambda u =-\Delta u +\lambda u$ and
 $B_\lambda \in B(X^p\cap V,L^p(\Omega))$, since $V \cap
 W^{1,q}(\Omega)$ is dense in $L^q(\Omega)$. Thus we obtain $B_\lambda =
 (-\Delta + \lambda)$. Moreover, $B_\lambda$ is injective, since so is
 $A_\lambda$.

As for the equivalence of (i)--(iv), we shall show (ii) $\Rightarrow$
 (i) $\Rightarrow$ (iii) $\Rightarrow$ (iv) $\Rightarrow$ (ii).
It is clear that (ii) $\Rightarrow$ (i) (and also (iv) $\Rightarrow$ (iii)).
We show (i) $\Rightarrow$ (iii). Assume (i), let $g\in L^p(\Omega)$ and
set $u:=A_\lambda^{-1}g\in W^{2,p}(\Omega)\cap V$. 
For all $v\in V \cap W^{1,q}(\Omega)$, we have
\begin{align*}
\int_\Omega gv \,\d x
=\langle A_\lambda u,v\rangle_V =
\int_\Omega \left( \nabla u\cdot\nabla v + \lambda uv \right)\,\d x
= \int_{\GammaN} (\partial_\nu u) v \,\d\cH^{n-1} +
\int_\Omega \left(-\Delta u+\lambda u\right)v \,\d x,
\end{align*}
which implies that $u\in X^p$ and $-\Delta u(x)+\lambda u(x)=g(x)$ for a.e.~$x \in \Omega$. Hence $u = B_\lambda^{-1} g$. Therefore,
$B_\lambda$ is surjective from $X^p \cap V$ into
 $L^p(\Omega)$. By the open mapping theorem, we obtain 
$(-\Delta +\lambda ) = B_\lambda \in \Isom (X^p\cap V,L^p(\Omega))$.

We next show (iii) $\Rightarrow$ (iv).
Under the condition (iii), it holds that 
$B_\lambda \in \Isom (X^p\cap V,L^p(\Omega))$,
in particular, $B_\lambda : X^p\cap V \to L^p(\Omega)$ is a
 Fredholm operator of index zero.
For arbitrary $\mu >0$, we find that $B_\mu =B_\lambda +(\mu
 -\lambda)$ is a Fredholm operator of index zero from $X^p\cap V$
 to $L^p(\Omega)$ as well, since $X^p \cap V$ is compactly embedded in
 $L^p(\Omega)$. Since $B_\mu$ is injective (i.e., $\dim \ker(B_\mu) =
 0$), we infer that $B_\mu$ is surjective, and hence $B_\mu$ also
 belongs to $\Isom (X^p\cap V,L^p(\Omega))$. Furthermore, for any $g \in
 L^p(\Omega)$ and $\mu > 0$, the element $u = (-\Delta + \mu)^{-1}g =
 B_\mu^{-1} g$ belongs to $X^p \cap V$. Hence $B_\mu u = g$,
 i.e., $A_\mu u = g$, which implies $A_\mu^{-1} g = u \in X^p$.
Thus (iv) follows.

It is obvious that (iv) implies (ii) by the definition of $X^p$.
Thus we have shown that all the conditions (i)--(iv) are equivalent to
 each other.
\end{proof}

\begin{Rem}
{\rm
Under \eqref{c3}, the assumptions for $\psi$ in \eqref{c2} is equivalent
 to $\psi \in X^p \cap V$. Indeed, let $\psi \in V$ satisfy $A\psi
 \in L^p(\Omega)$. Then as in the proof of Proposition \ref{P:K2}, one
 can check that $\psi \in W^{2,p}(\Omega)$. Moreover, by Green's
 formula, we find that
\begin{align*}
\int_\Omega A \psi v \, \d x
=
\langle A\psi, v \rangle_V 
&= \int_\Omega \nabla \psi \cdot \nabla v \, \d x + \sigma \int_\Omega
 \psi v \, \d x\\
&= \int_\Omega \left( - \Delta \psi + \sigma \psi \right) v \, \d x 
+ \int_{\GammaN} (\partial_\nu \psi) v \, \d \cH^{n-1}
\end{align*}
for all $v \in V \cap W^{1,q}(\Omega)$. Thus by the
 arbitrariness of $v$, we obtain $\partial_\nu \psi = 0$ \ $\cH^{n-1}$-a.e.~on $\GammaN$,
 whence follows $\psi \in X^p$.
}
\end{Rem}

Now, we are in position to state a regularity result for \eqref{vi} as a
corollary of Theorem \ref{th2}. This corollary will be used for proving
Theorem \ref{T:ex} in Section \ref{sec-td}.

\begin{Cor}[Regularity of solutions for variational inequalities of
 obstacle type]\label{cor1}
Assume that \eqref{A1}, \eqref{c1}, $f \in L^p(\Omega)$, $\psi \in X^p
 \cap V$ and \eqref{c3} are satisfied.
Let $u \in V$ be the unique element satisfying {\rm (a)--(h)}. Then it
 holds that
\begin{equation}\label{3}
u\in X^p \cap K_0, 
\quad
f \leq Au \leq f \vee \hat f \ \mbox{ a.e.~in } \Omega.
\end{equation}
\end{Cor}

\begin{proof}
By Theorem \ref{th2}, the unique element $u \in V$ satisfying (a)--(h)
 belongs to $K_2$. Hence, by Proposition \ref{P:K2}, one has $u
 \in W^{2,p}(\Omega)$. Since one can observe that $u = A_1^{-1}(Au + (1 -
 \sigma) u)$ and $Au + (1 - \sigma) u \in L^p(\Omega)$, by \eqref{c3}
 along with Proposition \ref{prop-GammaD}, it holds that $u \in X^p$. 
\end{proof} 

We next give a comparison theorem for variational inequalities of
obstacle type.

\begin{Th}[Comparison principle for variational inequalities of obstacle
 type]\label{T:comp-ell}
We suppose that \eqref{A1} and \eqref{c1} are satisfied. 
For $i=1,2$, let $f_i\in L^p(\Omega)$ and $\psi_i\in V$ be such that
 $A\psi_i \in L^p(\Omega)$ and set $K_0^i:=\{ v\in V \colon v\ge
 \psi_i~\mbox{a.e.~in~}\Omega \}$.
Let $u_i\in V$ be the unique solution of the variational inequality\/{\rm :}
\begin{align}\label{ivi}
u_i\in K_0^i,
\quad
a(u_i,v-u_i)\ge \langle f_i,v-u_i\rangle_V \ \mbox{ for all } \ v\in
K_0^i
\end{align}
for $i = 1,2$.
If $f_1\le f_2$ and $\psi_1\le \psi_2$ a.e.~in $\Omega$,
then $u_1\le u_2$ a.e.~in $\Omega$.
\end{Th}

This theorem will be used to prove Theorem \ref{T:comp}, a comparison
theorem for the evolutionary problem \eqref{is1}, with the aid of the
discretization argument. 

To prove Theorem \ref{T:comp-ell}, we prepare the following lemma.

\begin{Lem}\label{lem-cp-vi}
We suppose that \eqref{A1}, \eqref{c1} and \eqref{c2} are satisfied. 
Let $u \in V$ be the unique solution of {\rm (a)--(h)}.
Then it holds that $u\le w$ a.e.~in $\Omega$ for all $w\in K_0\cap K_1$
 satisfying $Aw \in L^p(\Omega)$.
\end{Lem}

\begin{proof}
We set $N:=\{x\in \Omega \colon w(x)<u(x)\}$ and $v:= u \wedge w$.
Since $w \in K_0$, by Lemma~\ref{lem-min} and \eqref{th-KS}, $v$ satisfies 
\begin{align*}
v\in K_0,
\quad
v= \begin{cases}
    u&\mbox{a.e.~in }\Omega\setminus N,\\
    w&\mbox{a.e.~in }N,
   \end{cases}
\quad
\nabla v=\begin{cases}
	  \nabla u&\mbox{a.e.~in }\Omega\setminus N,\\
\nabla w&\mbox{a.e.~in }N.
	 \end{cases}
\end{align*}
Substituting $v$ into the variational inequality (b) of Proposition~\ref{th1},
we have 
\begin{align*}
0\le \langle Au-f,v-u\rangle_V 
=\int_N (Au-f)(w-u)\,\d x.
\end{align*}
Since $Au - f \geq 0$ and $w-u<0$ a.e.~in $N$, it follows that
$Au=f$ a.e.~in $N$. Here we note that
\begin{align*}
v-u\in V,
\quad
v-u=\begin{cases}
     0&\mbox{a.e.~in }\Omega\setminus N,\\
     w- u <0&\mbox{a.e.~in }N,
    \end{cases}
\quad
\nabla (v-u)=\begin{cases}
	      0&\mbox{a.e.~in }\Omega\setminus N,\\
	      \nabla (w-u)&\mbox{a.e.~in }N.
	     \end{cases}
\end{align*}
From the fact that $u-v=(u-w)_+$ and $Aw \in L^p(\Omega)$, we obtain
\begin{align*}
0 \leq a(u-v,u-v)
&=a(u-w, u-v)
=
\langle A(u-w),u-v\rangle_V \\
&=
\int_N (Au-Aw)(u-w)\,\d x
=
\int_N (f-Aw)(u-w)\,\d x.
\end{align*}
Since $Aw \geq f$ (by $w \in K_1$) and $u > w$ a.e.~in $N$, we conclude
 that $a(u-v,u-v) = 0$, whence $u=v$ (hence $u \le w$) a.e.~in
 $\Omega$.
\end{proof}

The lemma above will also play a crucial role for identifying the
limit of each solution $u = u(x,t)$ for \eqref{is1}--\eqref{ic} as $t
\to \infty$ in a proof of Theorem \ref{T:asp}.

We are now ready to prove Theorem \ref{T:comp-ell}.

\begin{proof}[Proof of Theorem \ref{T:comp-ell}]
 By assumption, we find that $K^2_0 \subset K^1_0$ and $K^2_1 \subset
 K^1_1$. Moreover, $u_2$ belongs to both $K^1_0$ and $K^1_1 := \{v \in V
 \colon Av \geq f_1 \ \mbox{ in } V'\}$, and $A u_2 \in
 L^p(\Omega)$ as well. By Lemma \ref{lem-cp-vi}, we conclude that $u_1 \leq u_2$
 a.e.~in $\Omega$.
\end{proof} 

\section{Reduction to an evolution equation and the uniqueness of
 solution}\label{sec-uni}

In this section, we first reduce the problem \eqref{is1}--\eqref{ic} to
the Cauchy problem for a nonlinear evolution equation in $L^2(\Omega)$
with the aid of convex analysis. Then we shall prove
Theorem~\ref{th-uni} on the uniqueness of solution.

 Let us begin with reformulating \eqref{eqn} as a parabolic
 \emph{inclusion} with a multivalued nonlinear operator acting on the
 time derivative of $u(x,t)$. Let $\alpha : \mathbb R \to 2^{\mathbb R}$ be
 given by 
\begin{equation}\label{alpha}
 \alpha(s) = \begin{cases}
	      \{0\} \quad &\mbox{ if } s > 0,\\
	      (-\infty,0] \quad &\mbox{ if } s = 0
	     \end{cases}
 \end{equation}
 with the domain $D(\alpha) = [0,\infty)$. 
 Then $s + \alpha(s)$ is the (multi-valued) inverse mapping of the
 function $(s)_+$, and it can be also represented by
 $$
 \alpha(s) = \partial \I (s) \quad \mbox{ for } \ s \geq 0,
 $$
 where $\I$ denotes the indicator function over
 the set $[s \geq 0] := \{ s \in \mathbb R \colon s \geq 0\}$
 and $\partial$ means the \emph{subdifferential} in the sense of
 convex analysis (see, e.g.,~\cite{Bre73} and also \eqref{subdif} below with
 $H = \R$). Then \eqref{is1} can be reformulated as a
 \emph{doubly nonlinear}-type PDE,
 \begin{equation}\label{dnp}
  \partial_t u + \alpha(\partial_t u) \ni \Delta u + f \ \mbox{ in } Q.
 \end{equation}

We next reduce the PDE \eqref{dnp} to an \emph{evolution equation}. To
this end, define a functional $\phi : L^2(\Omega) \to [0,\infty]$ by
$$
\phi(v):=
\begin{cases}
\frac 1 2 \int_\Omega |\nabla v|^2\,\d x
&\mbox{if } v\in V,\\
+\infty 
& \mbox{if } v\in L^2(\Omega)\setminus V
\end{cases}
$$
with the \emph{effective domain} $D(\phi) := \{v \in L^2(\Omega) \colon
\phi(v) < +\infty\} = V$.
Then  we observe that:
\begin{Lem}\label{lem-mm}
The functional $\phi$ is convex and lower semicontinuous
in $L^2(\Omega)$. In particular, if \eqref{A1condition} is satisfied,
 then the subdifferential operator $\partial \phi$ of $\phi$ {\rm (}in
 $L^2(\Omega)${\rm )} is characterized as
\begin{align*}
D(\partial \phi)= X \cap V,\quad
\partial \phi(v)= -\Delta v\quad \mbox{ for } \ v\in X\cap V,
\end{align*}
where $\Delta$ stands for the Laplace operator from $D(\Delta) = X
 \cap V$ into $L^2(\Omega)$ as in Proposition \ref{prop-GammaD}.
\end{Lem}

Here let us recall the definition of the \emph{subdifferential operator}
$\partial \varphi : H \to H$ of a proper, lower semicontinuous and
convex functional $\varphi$ defined on a Hilbert space $H$,
\begin{equation}\label{subdif}
\partial \varphi(u) := \left\{
\xi \in H \colon \varphi(v) - \varphi(u) \geq (\xi, v - u)_H \ \mbox{
for all } v \in D(\varphi)
\right\}
\quad \mbox{ for } \ u \in D(\varphi),
\end{equation}
where $(\cdot,\cdot)_H$ stands for the inner product in $H$ and
$D(\varphi) := \{w \in H \colon \varphi(w) < +\infty\}$,
with domain $D(\partial \varphi) := \{w \in D(\varphi) \colon \partial
\varphi (w) \neq \emptyset\}$. It is well known that $\partial \varphi$
is a (possibly multivalued) maximal monotone operator in $H$ (see,
e.g.,~\cite{Bre73} for more details).
\begin{proof}[Proof of Lemma \ref{lem-mm}]
We note that the restriction $\phi_0 := \phi|_V$ of $\phi$ onto
 $V$ is Fr\'echet differentiable and the derivative $\phi_0'$ of
 $\phi_0$ satisfies
\begin{equation}\label{4}
\langle \phi_0'(u), z \rangle_V = \int_\Omega \nabla u \cdot
      \nabla z \, \d x \quad \mbox{ for all } \ z \in V.
\end{equation}
 Now, let $u \in D(\partial \phi)$ and $\xi \in \partial \phi(u) \subset
 L^2(\Omega)$. From the definition of subdifferentials, we find that $\partial
 \phi(u) \subset \partial \phi_0(u) = \{\phi_0'(u)\}$ for all $u \in
 D(\partial \phi) \subset D(\phi) = V$. Hence $\xi = \phi_0'(u)$, i.e.,
 $\partial \phi(u) = \{\phi_0'(u)\}$ and $\phi_0'(u) \in L^2(\Omega)$;
 here and henceforth, we simply write $\partial \phi(u) =
 \phi_0'(u)$. It follows that
$$
A_1 u = u + \phi_0'(u) = u + \xi \in L^2(\Omega). 
$$
Moreover, by \eqref{A1condition} along with Proposition
 \ref{prop-GammaD}, we deduce that $u = A_1^{-1} (u + \xi) \in X
 \cap V$ and that $u + \xi = A_1 u = - \Delta u + u$.
Therefore we deduce that $\partial \phi(u) = \phi_0'(u) = -\Delta u$ and
 $D(\partial \phi) \subset X \cap V$. On the other hand, it is clear
 that $X \cap V \subset D(\partial \phi)$, and hence, $D(\partial \phi)
 = X \cap V$.
\end{proof} 

Therefore the initial-boundary value problem for \eqref{dnp} equipped
with \eqref{bc-D}--\eqref{ic} can
be rewritten as the Cauchy problem for an evolution equation in
$L^2(\Omega)$ of $u(t) := u(\cdot, t)$,
\begin{equation}\label{ee2}
 \partial_t u(t) + \partial \I (\partial_t u(t)) + \partial \phi(u(t))
  \ni f(t) \ \mbox{ in } L^2(\Omega), \quad 0 < t < T, \quad u(0) = u_0,
\end{equation}
where $f(t) := f(\cdot,t)$ and $\partial \I$ denotes the
subdifferential operator in $L^2(\Omega)$ of the functional $\I :
L^2(\Omega) \to [0,\infty]$ defined by
$$ 
\I(v) = \begin{cases}
	 0 &\mbox{ if } v \geq 0 \ \mbox{ a.e.~in } \Omega,\\
	 \infty &\mbox{ otherwise}
	\end{cases}
	\quad \mbox{ for } \ v \in L^2(\Omega).
$$
We note that $\partial \I(v) = \alpha(v(\cdot))$ for $v \in
L^2(\Omega)$, where $\alpha(\cdot)$ is a multivalued function given by
\eqref{alpha}, and $D(\partial \I) = \{ v \in L^2(\Omega) \colon v \geq
0 \mbox{ a.e.~in } \Omega\}$ (see, e.g.,~\cite{Bre73}). 

Here and henceforth, for simplicity, we use the same notation $\I$ for
the indicator function over $[0,+\infty)$ defined on $\mathbb R$ as well
as for that over the set $\{v \in L^2(\Omega)
\colon v \geq 0 \ \mbox{ a.e.~in } \Omega\}$ defined on $L^2(\Omega)$,
unless any confusion may arise. Moreover, the subdifferential operators
of the both indicator functions are also denoted by $\partial \I$.

Strong solutions of \eqref{ee2} are defined as follows:

\begin{Def}[Strong solution of \eqref{ee2}]\label{D:sol}
For given $f\in L^2(Q_T)$ and $u_0\in L^2(\Omega)$, a function $u \in
 C([0,T];L^2(\Omega))$ is called a \emph{strong solution} of
 \eqref{is1}--\eqref{ic} on $[0,T]$, if the following conditions are
 satisfied\/{\rm :}
\begin{itemize}
\item
     $u \in W^{1,2}(0,T;L^2(\Omega)) \cap L^2(0,T;X \cap
     V)$\/{\rm ;}
\item It holds that
     \begin{equation}\label{ee2-EQ}
      \partial_t u(t) + \partial \I(\partial_t u(t)) + \partial \phi(u(t))
     \ni f(t)
      \ \mbox{ in } L^2(\Omega) \ \mbox{ for a.e. } t \in (0,T)\/{\rm ;}
      \end{equation}
\item
     $u(0) = u_0$, 
\end{itemize}
where the functionals $\I$ and $\phi$ on $L^2(\Omega)$ are defined
 as above.
\end{Def}

\begin{Prop}[Equivalence of solutions]
The notion of strong solutions for \eqref{ee2} is equivalent to that for
 \eqref{eqn}--\eqref{ic} defined by Definition {\rm \ref{def-ss}}.
\end{Prop}

\begin{proof}
Since $\alpha(s)$ is the inverse mapping of $s \mapsto (s)_+$, one observes
 that \eqref{eqn} is equivalent to \eqref{dnp} at each $(x,t) \in
 Q$. Moreover, due to Lemma \ref{lem-mm}, for each strong solution $u$
 of \eqref{ee2}, $u(x,t)$ satisfies \eqref{dnp} a.e.~in $\Omega \times
 (0,T)$. Conversely, let $u$ be a strong solution of \eqref{eqn} in the
 sense of Definition \ref{def-ss}. Then from the regularity condition {\rm
 (ii)} of Definition \ref{def-ss}, the right-hand-side of
 the inclusion 
$$
\alpha(\partial_t u) \ni \Delta u + f - \partial_t u
$$
belongs to $L^2(\Omega)$ for a.e.~$t \in (0,T)$.
Moreover, recalling that
$$
\partial \I(v) = \{
 \xi(\cdot) \in L^2(\Omega) \colon \xi(x) \in \alpha(v(x)) \mbox{ for a.e. } x \in \Omega \}
$$
(see above), the evolution equation \eqref{ee2-EQ} holds in
 $L^2(\Omega)$ for a.e.~$t \in (0,T)$.
\end{proof}

We next provide a chain-rule for the function $t \mapsto \phi(u(t))$,
which is derived from a standard theory on subdifferential calculus and
which will be used frequently to derive energy estimates in later
sections.

\begin{Lem}\label{lemforuq}
We suppose that $u\in W^{1,2}(0,T;L^2(\Omega)) \cap L^2(0,T;X \cap
 V)$. Then we have\/{\rm :}
\begin{enumerate}[{\rm (i)}]
\item
the function 
$$
t\mapsto \phi(u(t)) = \dfrac 1 2 \int_\Omega |\nabla u(t,x)|^2\,\d x
$$
belongs to $W^{1,1}(0,T)$\/{\rm ;}
\item
for a.e.~$t \in (0,T)$, it holds that
$$
\dfrac 1 2 \frac{\d}{\d t}\int_\Omega |\nabla u|^2\,\d x
= \dfrac{\d}{\d t} \phi(u(t))
= \left( \partial \phi(u(t)), \partial_t u(t) \right)_{L^2(\Omega)}
=-\int_\Omega \partial_t u \,\Delta u\,\d x,
$$
where $(\cdot,\cdot)_{L^2(\Omega)}$ denotes the inner product of
	     $L^2(\Omega)$\/{\rm ;}
\item
$u\in C([0,T];V)$.
\end{enumerate}

\end{Lem}

\begin{proof}
Thanks to Lemma~3.3 of \cite{Bre73}, the assertions (i) and (ii) follow
 immediately. Concerning (iii), since $u$ belongs to $L^\infty(0,T;V)$
 and $C([0,T];L^2(\Omega))$, by exploiting Lemma 8.1 of~\cite{LM}, one
 finds that $u$ is continuous on $[0,T]$ with respect to the weak
 topology of $V$. On the other hand, $t \mapsto \|u(t)\|_V$ is continuous on
 $[0,T]$ by (i). Therefore from the uniform convexity of $\|\cdot\|_V$,
 we deduce that $t \mapsto u(t)$ is continuous on $[0,T]$ with respect
 to the strong topology of $V$.
\end{proof}

Before proceeding to a proof of Theorem \ref{th-uni}, let us note that
\begin{align}\label{ab}
|a_+-b_+|^2\le |a_+-b_+||a-b| = (a_+-b_+)(a-b)
\quad
\mbox{ for all } \ a,b\in \R,
\end{align}
since the function $s \mapsto s_+ = s \vee 0$ is nondecreasing and 
non-expansive, that is, $|a_+-b_+|\le |a-b|$.
Now, we are ready to prove Theorem \ref{th-uni}.

\begin{proof}[Proof of Theorem~\ref{th-uni}]
For each $i = 1,2$, let $u_i$ be a strong solution of \eqref{is1}--\eqref{ic}
 with $u_0 = u_{0,i} \in V$ and $f = f_i \in L^2(0,T;L^2(\Omega))$ and
 set $u = u_1 - u_2$.
Then due to Lemma~\ref{lemforuq}, we have
\begin{align*}
\dfrac 12 \dfrac{\d}{\d t}
\int_\Omega |\nabla u|^2 \, \d x
&= - \int_\Omega \partial_t u \, \Delta u \,\d x\\
&= - \int_\Omega \left[ (\Delta u_1+f_1)_+-(\Delta u_2+f_2)_+ \right]
\left[ (\Delta u_1+f_1)-(\Delta u_2+f_2) - f_1 + f_2 \right] \,\d x\\
&\stackrel{\eqref{ab}}{\le} - \dfrac 1 2 \int_\Omega \left|
(\Delta u_1+f_1)_+-(\Delta u_2+f_2)_+ \right|^2 \,\d x \\
&\qquad + \dfrac 1 2 \int_\Omega |f_1 - f_2|^2 \, \d x
\quad \mbox{ for a.e. } t\in (0,T),
\end{align*}
which implies that 
$$
\int_\Omega |\partial_t u_1 - \partial_t u_2|^2 \, \d x
+ \dfrac{\d}{\d t} \int_\Omega |\nabla u_1 - \nabla u_2|^2 \, \d x
\leq \int_\Omega |f_1 - f_2|^2 \, \d x
\quad \mbox{ for a.e. } t\in (0,T).
$$
Integrate both sides with respect to $t$ to obtain
\begin{align}
\int^T_0 \|\partial_t u_1(t)-\partial_t u_2(t)\|_{L^2(\Omega)}^2 \, \d t
+ \sup_{t \in [0,T]} \|\nabla u_1(t)-\nabla u_2(t)\|_{L^2(\Omega)}^2
\qquad \nonumber \\
\leq 2 \left( \|\nabla u_{0,1}-\nabla u_{0,2}\|_{L^2(\Omega)}^2
+ \int^T_0 \|f_1(t)-f_2(t)\|_{L^2(\Omega)}^2 \, \d t \right).\label{conti-dep}
\end{align}
In particular, if $u_{0,1} = u_{0,2}$ and $f_1 = f_2$, then $u_1$
 coincides with $u_2$ a.e.~in $Q_T$.
Consequently, the solution of \eqref{is1}--\eqref{ic} is unique.
\end{proof}

\begin{Cor}[Continuous dependence of solutions on data]
 For each $T > 0$ and $i = 1,2$, let $u_i$ be the strong solution of
 \eqref{is1}--\eqref{ic} on $[0,T]$ with $u_0 = u_{0,i} \in V$ and $f =
 f_i \in L^2(Q_T)$. Then \eqref{conti-dep} holds true.
\end{Cor}

\section{Existence of solutions and comparison principle}\label{sec-td} 

In this section, we shall prove Theorem \ref{T:ex} on the existence of
solutions for \eqref{is1}--\eqref{ic}.
Let $T>0$ be fixed. We denote by $\tau$ a division
$\{t_0, t_1, \ldots, t_m \}$ of the interval $[0,T]$ given by
$$
0=t_0<t_1<\ldots <t_m=T,
\quad
\tau_k :=t_k-t_{k-1}
\ \mbox{ for } \ k=1,\ldots,m,
\quad |\tau|:=\max_{k=1,\ldots,m} \tau_k.
$$
We shall construct $u_k \in X \cap V$ (for $k = 1,2,\ldots,m$), which is
an approximation of $u(t_k)$ for a solution $u$ of \eqref{eqn} by the
backward-Euler scheme
\begin{equation}\label{star}
\dfrac{u_k-u_{k-1}}{\tau_k} = \left( \Delta u_k + f_k \right)_+ \ \mbox{
a.e.~in } \Omega,
\end{equation}
where $f_k\in L^2(\Omega)$ is given by
$$
f_k := \dfrac 1 {\tau_k} \int^{t_k}_{t_{k-1}} f(\cdot,s) \,\d s.
$$
For given $u_0\in V$, we shall inductively define $u_k \in V$ for $k =
1,2,\ldots, m$ as a (global) minimizer of the functional
\begin{equation}
J_k(v):=\frac{1}{2\tau_k}\int_\Omega |v|^2\,\d x
+\frac{1}{2}\int_\Omega |\nabla v|^2\,\d x
- \left\langle \dfrac{u_{k-1}}{\tau_k} + f_k ,v \right\rangle_V \label{Jk1}
\quad \mbox{ for } \ v \in V
\end{equation}
subject to
\begin{equation}
v \in K_0^k:=\{ v\in V \colon v \ge u_{k-1} \ \mbox{ a.e.~in } \Omega
 \}.\label{Jk2}
\end{equation}

\begin{Rem}[Derivation of the discretized problems]
{\rm
 The minimization problems with constraints stated above can be
 also derived from a discretization of the evolution equation \eqref{ee2},
 which is equivalent to \eqref{is1} (see Remark \ref{R:def} and
 Section \ref{sec-uni}). 
A natural time-discretization of \eqref{ee2} may be given as
\begin{equation}\label{td2}
\dfrac{u_k-u_{k-1}}{\tau_k} + \partial_V \I \left(
 \dfrac{u_k-u_{k-1}}{\tau_k} \right) - \Delta u_k \ni f_k
\ \mbox{ in } V'
\end{equation}
(here $\partial_V$ stands for the subdifferential of the functional $\I$
 restricted onto $V$),
which is an Euler-Lagrange equation for the functional
$$
E_k(v) := \dfrac{1}{2\tau_k} \int_\Omega |v|^2 \, \d x
+ \I \left( \dfrac{v - u_{k-1}}{\tau_k} \right)
+ \dfrac 1 2 \int_\Omega |\nabla v|^2 \, \d x 
- \left\langle \dfrac{u_{k-1}}{\tau_k} + f_k, v
 \right\rangle_V \ \mbox{ for } \ v \in V.
$$
Indeed, since $E_k$ is coercive, lower semicontinuous and convex in $V$,
 $E_k$ admits a global minimizer $u_k$ over $V$, and moreover, $u_k$ solves
 \eqref{td2} in $V'$. Here we note that the minimization of $E_k$ over
 $V$ is equivalent to that of $J_k$ over $K_0^k$ from the fact that
$$
\I \left( \dfrac{v-u_{k-1}}{\tau_k} \right) 
= I_{[\,\cdot\, \geq u_{k-1}]}(v)
:= \begin{cases}
    0 &\mbox{ if } \ v \geq u_{k-1} \ \mbox{ a.e.~in } \Omega,\\
    \infty &\mbox{ otherwise}
   \end{cases}
 \quad \mbox{ for } \ v \in L^2(\Omega).
$$
}
\end{Rem}

Applying the regularity theory established in Section \ref{sec-vi}, one
can actually obtain the unique minimizer $u_k$ of $J_k$ over $K_0^k$ for each
$k$. More precisely, we obtain the following theorem, where we set
$$
g_k:=\frac{u_k-u_{k-1}}{\tau_k}-\Delta u_k-f_k.
$$
\begin{Lem}[Existence and regularity of minimizers]\label{th-uk}
For given $u_0\in V$ and each $k = 1,2,\ldots, m$, there exists a unique
 element $u_k \in K_0^k$ which minimizes \eqref{Jk1} subject to \eqref{Jk2}.
Moreover, for each $k = 1,2,\ldots, m$, the minimizer $u_k$ belongs to
 $X$ and fulfills \eqref{star}, that is,
\begin{align}
&u_k - u_{k-1} \geq 0 \ \mbox{ a.e.~in } \ \Omega,\label{n40}\\
&g_k \geq 0 \ \mbox{ a.e.~in } \ \Omega,
&\label{c-K1}\\
&\langle g_k, u_k-u_{k-1} \rangle_V =0,
&\label{c-c}
\end{align}
Furthermore, one has
\begin{align}
&\langle g_k, v-u_k \rangle_V \ge 0
\quad \mbox{ for all }\ v\in K_0^k,
&\label{c-b}\\
&\langle g_k+f_k+\Delta u_{k-1}, v-u_k \rangle_V \ge 0
\quad
\mbox{ for all } \ v\in K_1^k,
&\label{c-d}
\end{align}
where $\Delta$ is the Laplace operator from $X \cap V$ into
 $L^2(\Omega)$ {\rm (}see Section \ref{sec-uni}{\rm )} and the set
 $K^k_1$ is given by
\begin{align*}
 K_1^k:=\left\{v\in V \colon \frac{v-u_{k-1}}{\tau_k}-\Delta v-f_k\ge
  0 \ \mbox{\rm in } V' \right\}.
\end{align*}
In addition, suppose that the conditions \eqref{c1} and \eqref{c3} hold. If $u_0\in X^p \cap V$
 and $\{f_k\}_{k=1}^m\subset L^p(\Omega)$, 
then $\{u_k\}_{k=1}^m\subset X^p$ and it holds that
\begin{align}\label{c-K2}
0\le g_k
\le \left(-\Delta u_{k-1}-f_k \right)_+
\ \mbox{ a.e.~in } \Omega
\ \mbox{ for each } \ k=1,\ldots,m.
\end{align}
\end{Lem}

\begin{proof}
 Let us start with $k = 1$.
 By setting $\sigma =1/\tau_k > 0$ (i.e., $Au = A_\sigma u = u/\tau_k -
 \Delta u$), $f=f_k+u_{k-1}/\tau_k \in
 L^2(\Omega)$ and $\psi =u_{k-1} \in X \cap V$, one can write
 $K_0^k=K_0$ and $J_k(v)=J(v)$ for $v \in V$ with $K_0$ and $J(v)$
 defined by \eqref{K0} and \eqref{a2} along with \eqref{a1}. Hence,
 one can apply Proposition \ref{th1} and Theorem \ref{th2} to the
 minimization problem of $J_k$ over $K_0^k$. Then the minimizer
 $u_k \in K_0^k$ of $J_k$ over $K_0^k$ uniquely exists, and
 furthermore, \eqref{c-K1}--\eqref{c-d} follow immediately from the fact
 $u_k \in
 K_1^k$, (b), (c) and (d) of Proposition \ref{th1}, respectively. Moreover,
 by virtue of \eqref{A1condition} and Corollary \ref{cor1}, one can
 deduce that $u_k \in X$.
Repeating the argument above for $k = 2,3,\ldots,m$, we can inductively
 obtain $u_k
 \in K_0^k \cap X$ satisfying \eqref{c-K1}--\eqref{c-d} for each $k =
 2,3,\ldots,m$.

 Finally, if $u_0 \in X^p \cap V$, $f_k \in L^p(\Omega)$ and \eqref{c3}
 is satisfied for $p$ satisfying \eqref{c1}, by Corollary \ref{cor1}, we
 can assure that $u_k \in X^p \cap K_0^k$ and
$$
f_k+\dfrac{u_{k-1}}{\tau_k} \leq \dfrac{u_k}{\tau_k} - \Delta u_k 
\leq \left(f_k + \dfrac{u_{k-1}}{\tau_k} \right) \vee 
\left( \dfrac{u_{k-1}}{\tau_k} - \Delta u_{k-1}\right)
\quad \mbox{ for a.e. } x \in \Omega,
$$
which is equivalent to \eqref{c-K2}.
\end{proof}

\begin{proof}[Proof of Theorem \ref{T:ex}]
Let us define the \emph{piecewise linear interpolant} $u_\tau \in
W^{1,\infty}(0,T;X\cap V)$ of $\{u_k\}$ and the \emph{piecewise
constant interpolants} $\bar{u}_\tau \in L^\infty (0,T;X\cap V)$ and
$\bar{f}_\tau \in L^\infty (0,T;L^2(\Omega))$ of $\{u_k\}$ and
$\{f_k\}$, respectively, by
\begin{alignat*}{4}
&u_\tau(t):=u_{k-1}+\frac{t-t_{k-1}}{\tau_k}(u_k-u_{k-1})
\quad && \mbox{ for } \ t\in [t_{k-1},t_k] \ \mbox{ and } \ k=1,\ldots,m,\\
&\bar{u}_\tau(t):=u_k, \quad \bar{f}_\tau(t):=f_k
\quad && \mbox{ for } \ t\in (t_{k-1},t_k] \ \mbox{ and } \ k=1,\ldots,m.
\end{alignat*}
By summing up \eqref{c-c} for $k=1,\ldots,\ell$ with an arbitrary
natural number $\ell \leq m$, we have
\begin{align}
\lefteqn{
\sum_{k = 1}^\ell \tau_k
\left\|\dfrac{u_k - u_{k-1}}{\tau_k}\right\|_{L^2(\Omega)}^2 
+ \dfrac 1 2 \int_\Omega |\nabla u_\ell|^2 \, \d x
- \dfrac 1 2 \int_\Omega |\nabla u_0|^2 \, \d x
}\nonumber \\
&\leq \sum_{k=1}^\ell \tau_k \left( f_k,
 \dfrac{u_k-u_{k-1}}{\tau_k} \right)_{L^2(\Omega)} 
\leq \dfrac 1 2 \sum_{k=1}^\ell \tau_k \|f_k\|_{L^2(\Omega)}^2
+ \dfrac 1 2 \sum_{k=1}^\ell \tau_k \left\|
 \dfrac{u_k-u_{k-1}}{\tau_k}\right\|_{L^2(\Omega)}^2,\label{ei-dsc}
\end{align}
which implies
\begin{align*}
\int_0^t
\|\partial_t u_\tau(s)\|_{L^2(\Omega)}^2 \, \d s
+ \|\nabla \bar u_\tau(t)\|_{L^2(\Omega)}^2
\le
\|\nabla u_0\|_{L^2(\Omega)}^2
+\int_0^T \left\|\bar f_\tau(s)\right\|_{L^2(\Omega)}^2 \, \d s
\quad
\mbox{ for all } \ t \in [0,T].
\end{align*}
Hence, we obtain
\begin{align}
\left\| \partial_t u_\tau \right\|_{L^2(0,T;L^2(\Omega))}^2
 + \sup_{t \in [0,T]} \left\| \nabla \bar{u}_\tau(t)
 \right\|_{L^2(\Omega)}^2
 + \sup_{t \in [0,T]} \left\| \nabla u_\tau(t) \right\|_{L^2(\Omega)}^2
\qquad \nonumber\\
\le
C \left( \left\| \nabla u_0 \right\|_{L^2(\Omega)}^2
+\left\| \bar f_\tau\right\|_{L^2(0,T;L^2(\Omega))}^2 \right).
\label{est1}
\end{align}

Now, let us take a limit as $m \to \infty$ such that $|\tau| \to 0$ and
note that
\begin{equation}\label{c:f}
 \bar f_\tau \to f \quad \mbox{ strongly in } L^2(0,T;L^2(\Omega)).
\end{equation}
In particular, $\{\bar f_\tau\}$ is bounded in $L^2(0,T;L^2(\Omega))$.
Indeed, one can verify that 
$$
\|\bar f_\tau\|_{L^2(0,T;L^2(\Omega))} \leq \|f\|_{L^2(0,T;L^2(\Omega))}.
$$
From the uniform estimate \eqref{est1}, one can take a function  
$
u\in W^{1,2}(0,T;L^2(\Omega)) \cap L^\infty(0,T;V)
$
(in particular, $u \in C([0,T];L^2(\Omega))$ as well)
such that, up to a (non-relabeled) subsequence,  
\begin{alignat}{4}
u_\tau &\to u &&\mbox{ weakly in } W^{1,2}(0,T;L^2(\Omega)),\label{c:u2:12}\\
 & &&\mbox{ weakly star in }L^\infty(0,T;V),\label{c:uV:i}\\
 & &&\mbox{ strongly in }C([0,T];L^2(\Omega)),\label{c:u2:C}\\
\bar{u}_\tau &\to u &&\mbox{ weakly star in }L^\infty(0,T;V),\label{c:buV:i}\\
u_\tau(T) &\to u(T) \qquad &&\mbox{ weakly in } V.\label{c:uT:V}
\end{alignat}
Here, the weak and weak star convergence of $u_\tau$ and $\bar
u_\tau$ immediately follow from the uniform estimate \eqref{est1}.
Moreover, we also note that $u_\tau$ and $\bar u_\tau$ possess a common
limit function. Indeed, by a simple calculation, we observe that
\begin{align*}
 \|u_\tau(t) - \bar u_\tau(t)\|_{L^2(\Omega)}
&= \left|\dfrac{t_k-t}{\tau_k}\right|\left\| u_k - u_{k-1}
 \right\|_{L^2(\Omega)}\\
&\leq \left\| \dfrac{u_k - u_{k-1}}{\tau_k}\right\|_{L^2(\Omega)} \tau_k\\
&\stackrel{\eqref{ei-dsc}}\leq C |\tau|^{1/2}
\quad \mbox{ for all } \ t \in (t_{k-1}, t_k], \ k = 1,2,\ldots,m,
\end{align*}
which yields that
$$
\sup_{t \in [0,T]}  \|u_\tau(t) - \bar u_\tau(t)\|_{L^2(\Omega)} \leq
C|\tau|^{1/2} \to 0.
$$
Thus $u_\tau$ and $\bar u_\tau$ (weakly) converge to a common limit function.
Furthermore, since $V$ is compactly embedded in $L^2(\Omega)$, due to
Ascoli's compactness lemma along with \eqref{est1}, we obtain the strong convergence \eqref{c:u2:C}.
Since $u_\tau(T) = u_m$ is bounded in $V$ by \eqref{ei-dsc}, one can also derive
\eqref{c:uT:V} from \eqref{c:u2:C}. We further observe that $u(0) = u_0$.

We next estimate $\Delta \bar u_\tau$ in $L^2(0,T ; L^2(\Omega))$
by using \eqref{c-K2} and the assumption \eqref{f-hypo1}.
We first rewrite \eqref{c-K2} as
\begin{equation}\label{star2}
- \dfrac{u_k-u_{k-1}}{\tau_k} + f_k \leq - \Delta u_k \leq
\left(- \dfrac{u_k-u_{k-1}}{\tau_k} + f_k\right) \vee
\left(- \dfrac{u_k-u_{k-1}}{\tau_k} - \Delta u_{k-1} \right)
\ \mbox{ a.e.~in } \Omega.
\end{equation}
Since $(u_k - u_{k-1})/\tau_k \geq 0$ a.e.~in $\Omega$ by $u_k \in
K^k_0$, we observe by \eqref{f-hypo1} that
$$
\mbox{(The right-hand side of \eqref{star2})}
\leq f_k \vee (-\Delta u_{k-1})
\leq f^* \vee (-\Delta u_{k-1})
\ \mbox{ a.e.~in } \Omega,
$$
which also iteratively implies that
\begin{align*}
 - \Delta u_k &\leq f^* \vee (-\Delta u_{k-1})\\
 &\leq f^* \vee \left( f^* \vee (-\Delta u_{k-2}) \right)\\
 &= f^* \vee (-\Delta u_{k-2})
 \leq \cdots \leq f^* \vee (-\Delta u_0)
\quad \mbox{ a.e.~in } \Omega.
\end{align*}
Thus we obtain
$$
- \dfrac{u_k-u_{k-1}}{\tau_k} + f_k \leq - \Delta u_k \leq
f^* \vee (-\Delta u_0)
\quad \mbox{ a.e.~in } \Omega,
$$
which yields that
$$
\|\Delta u_k\|_{L^2(\Omega)}^2 \leq 
2 \left( \|f^*\|_{L^2(\Omega)}^2 + \|\Delta u_0\|_{L^2(\Omega)}^2 +
\left\|\dfrac{u_k - u_{k-1}}{\tau_k}\right\|_{L^2(\Omega)}^2
+ \|f_k\|_{L^2(\Omega)}^2 \right)
$$
for $k = 1,2,\ldots, m$. Hence we deduce that
\begin{align}
 \int^T_0 \|\Delta \bar u_\tau(t)\|_{L^2(\Omega)}^2 \, \d t
 &\leq 2 T \left( \|f^*\|_{L^2(\Omega)}^2 + \|\Delta
 u_0\|_{L^2(\Omega)}^2 \right)\nonumber\\
 &\quad + 2 \int^T_0 \left\| \partial_t u_\tau(t)
 \right\|_{L^2(\Omega)}^2 \, \d t
 + 2\int^T_0 \left\| \bar f_\tau (t) \right\|_{L^2(\Omega)}^2 \,\d t
 \leq C
\label{dot}
\end{align}
by using \eqref{est1} and \eqref{c:f}.

Exploiting Proposition \ref{prop-GammaD} with \eqref{A1condition}, we
see that $(I - \Delta) \in \Isom(X \cap V , L^2(\Omega))$, which
together with \eqref{dot} gives
$$
\int^T_0 \|\bar u_\tau(t)\|_{X}^2 \, \d t
\leq C \int^T_0 \left(
\|\Delta \bar u_\tau(t)\|_{L^2(\Omega)}^2
+ \|\bar u_\tau(t)\|_{L^2(\Omega)}^2 
\right)\, \d t
\leq C.
$$
Therefore we have, up to a (non-relabeled)subsequence,
\begin{alignat*}{4}
 \bar u_\tau &\to u \qquad &&\mbox{ weakly in } L^2(0,T;X),\\
 \Delta \bar u_\tau &\to \Delta u \qquad &&\mbox{ weakly in }
 L^2(0,T;L^2(\Omega)),
\end{alignat*}
which particularly implies $u(t) \in D(\Delta) = X \cap V$ for a.e.~$t
\in (0,T)$. Therefore the piecewise constant interpolant
$\bar g_\tau$ of $\{g_k\}$ defined by
$$
\bar g_\tau(t) := g_k \stackrel{\eqref{c-K1}}= \dfrac{u_k - u_{k-1}}{\tau_k} - \Delta u_k - f_k
 \quad \mbox{ for } \ t \in (t_{k-1},t_k]
$$
converges to
\begin{equation}\label{g}
\partial_t u - \Delta u - f =: g
\end{equation}
weakly in $L^2(0,T;L^2(\Omega))$.

It remains to prove that $u$ solves \eqref{is1} for
a.e.~$(x,t) \in Q_T$. To this end, we recall the
evolution equation \eqref{ee2} equivalent to \eqref{is1}. Then it
suffices to check that
$$
\partial_t u \geq 0 \ \mbox{ a.e.~in } Q_T \quad \mbox{ and } \quad
-g(t) \in \partial \I(\partial_t u(t)) \ \mbox{ for a.e.~} t \in (0,T).
$$
To this end, we employ the so-called Minty's trick for maximal monotone
operators, since $\partial \I$ is maximal monotone in $L^2(\Omega)$.
\begin{Prop}[Demiclosedness of maximal monotone operators (see, e.g.,~\cite{Bre73,BCP,B})]\label{P:demiclo}
 Let $A : H \to H$ be a {\rm (}possibly multivalued{\rm )} maximal
 monotone operator defined on a Hilbert space $H$ equipped with a inner
 product $(\cdot,\cdot)_H$. Let $[u_n,\xi_n]$ be
 in the graph of $A$ such that $u_n \to u$ weakly in $H$ and $\xi_n \to
 \xi$ weakly in $H$. Suppose that 
$$
\limsup_{n \to +\infty} ( \xi_n, u_n )_H \leq ( \xi, u )_H. 
$$
Then $[u,\xi]$ belongs to the graph of $A$, and moreover, it holds that
$$
\lim_{n \to +\infty} ( \xi_n, u_n )_H = ( \xi, u )_H. 
$$
\end{Prop}

Note that $(u_k - u_{k-1})/\tau_k \geq 0$ a.e.~in $\Omega$.
For an arbitrary $w \in D(\I) = \{v \in L^2(\Omega) \colon v \geq 0
\mbox{ a.e.~in } \Omega\}$, substitute $v = w\tau_k + u_{k-1} \in K_0^k$ to
\eqref{c-b}. Then we see that
$$
0 \stackrel{\eqref{c-b}}{\geq} \langle - g_k, \, v - u_k \rangle_V
= \tau_k \left( -g_k, \, w - \dfrac{u_k - u_{k-1}}{\tau_k}
\right)_{L^2(\Omega)},
$$
which together with the arbitrariness of $w \in D(\I)$ and the
definition of $\I$ implies that
$$
- g_k \in \partial \I \left( \dfrac{u_k - u_{k-1}}{\tau_k} \right),
\quad \mbox{ i.e., }
- \bar g_\tau(t) \in \partial \I \left( \partial_t u_\tau(t) \right).
$$
Moreover, for $k = 1,2,\ldots,m$, we find by \eqref{c-K1} that 
$$
\left( -g_k, \dfrac{u_k - u_{k-1}}{\tau_k}
\right)_{L^2(\Omega)} 
\leq - \left\| \dfrac{u_k - u_{k-1}}{\tau_k} \right\|_{L^2(\Omega)}^2
- \dfrac{\phi(u_k) - \phi(u_{k-1})}{\tau_k}
+ \left( f_k, \dfrac{u_k - u_{k-1}}{\tau_k} \right)_{L^2(\Omega)},
$$
which leads us to get
\begin{align*}
\int^T_0 \left( -\bar g_\tau(t), \partial_t u_\tau(t)
\right)_{L^2(\Omega)} \, \d t
&\leq - \int^T_0 \left\| \partial_t u_\tau(t)
 \right\|_{L^2(\Omega)}^2 \,\d t
- \phi(u_\tau(T)) + \phi(u_0)\\
 & \qquad + \int^T_0 \left( \bar f_\tau(t), \partial_t u_\tau(t)
 \right)_{L^2(\Omega)}  \, \d t. 
\end{align*}
Taking a limsup as $|\tau| \to 0$ in both sides, exploiting the weak
lower semicontinuity of norms and the functional $\phi(\cdot)$, and
recalling Lemma \ref{lemforuq}, we conclude that
\begin{align}
\limsup_{|\tau| \to 0}
\int^T_0 \left( -\bar g_\tau(t), \partial_t u_\tau(t)
\right)_{L^2(\Omega)} \, \d t
&\leq - \int^T_0 \left\| \partial_t u (t) \right\|_{L^2(\Omega)}^2
 \,\d t - \phi(u(T)) + \phi(u_0)\nonumber\\
 & \qquad + \int^T_0 \left( f(t), \partial_t u(t) \right)_{L^2(\Omega)}
 \, \d t\nonumber\\
&=
\int^T_0 \left( - \partial_t u(t) + \Delta u(t) + f(t),
 \partial_t u(t) \right)_{L^2(\Omega)} \, \d t \nonumber\\
&\stackrel{\eqref{g}}=  \int^T_0 \left( - g(t), \partial_t u(t) \right)_{L^2(\Omega)} \, \d t.\label{limsup}
\end{align}
Consequently, by virtue of the (weak) closedness of maximal monotone
operators (see Proposition \ref{P:demiclo} above), it follows that
$\partial_t u(t) \in D(\partial \I)$, i.e., $\partial_t u(t) \geq 0$ a.e.~in $\Omega$, and $- g(t) \in \partial \I(\partial_t u(t))$ for a.e.~$t \in
(0,T)$. Therefore $u$ solves \eqref{ee2}, and hence, $u$ is a strong
solution of \eqref{is1}--\eqref{ic}. Thus Theorem \ref{T:ex} has been
proved.
\end{proof}

\begin{Rem}
{\rm
To prove that $u$ is a strong solution of \eqref{eqn}--\eqref{ic},
it is possible to show the conditions (V1)-(V6) of Theorem \ref{th-ws}
 in \S \ref{od},
instead of the last argument of the proof of Theorem \ref{T:ex}.
Actually, (V2) and (V3) directly follow from \eqref{n40} and \eqref{c-K1}
by taking limit of $|\tau |\to 0$, respectively.
The condition (V4) follows from the estimate \eqref{limsup},
since the left-hand side of \eqref{limsup} is zero
and the right-hand side is non-positive.
}
\end{Rem}

Due to Theorem \ref{th-uni}, the limit of $\{u_\tau\}$ and $\{\bar
u_\tau\}$ is unique, whence they converge along the full sequence.
\begin{Cor}
Sequences $\{u_\tau\}$ and $\{\bar u_\tau\}$ converge to the unique
 solution $u$ of \eqref{eqn}--\eqref{ic} as $|\tau| \to 0_+$.
\end{Cor}

We next prove Theorem \ref{T:comp}.

\begin{proof}[Proof of Theorem \ref{T:comp}]
 Let $u^1$ and $u^2$ be strong solutions of \eqref{is1} with $u_0 =
 u_0^i$ and $f = f^i$ for $i = 1,2$, respectively. By the uniqueness of solutions (see
 Theorem \ref{th-uni}) and the construction of solutions discussed so
 far, one can take discretized solutions $\{u_k^i\}$ for $i = 1,2$ such that
 the piecewise linear interpolant $u_\tau^i$ of $\{u_k^i\}$ converges to
 $u^i$ strongly in $C([0,T];L^2(\Omega))$ as $|\tau| \to 0$, and they
 solve the variational inequalities
$$
\begin{cases}
\ \displaystyle u_k^i \in K_0(u_{k-1}^i) 
 := \{ v \in V
 \colon v \geq u_{k-1}^i \ \mbox{ a.e.~in } \Omega\},\\
\ \displaystyle a_k(u_k^i , v-u_k^i) 
 \geq \int_\Omega \left( f_k^i + u_{k-1}^i/\tau_k \right) (v - u_k^i) \,
 \d x \quad \mbox{ for all } \  v \in K_0(u_{k-1}^i),
\end{cases}
$$
where $a_k(\cdot,\cdot)$ stands for the bilinear form given by
$$
a_k(u,v) = \int_\Omega \nabla u \cdot \nabla v \,\d x 
 + \dfrac 1 {\tau_k} \int_\Omega u v \,\d x
\quad \mbox{ for } \ u, v \in V.
$$
By iteratively applying the comparison theorem for
 elliptic variational inequalities (see Theorem \ref{T:comp-ell}), from
 the fact that $f^1 \leq f^2$ a.e.~in $Q_T$ and $u_0^1 \leq u_0^2$
 a.e.~in $\Omega$, one can deduce that
$$
 u_k^1 \leq u_k^2 \ \mbox{ a.e.~in } Q_T 
\quad \mbox{ for all } \ k = 1,2,\ldots,m,
$$
which also implies $u^1_\tau(t) \leq u^2_\tau(t)$ a.e.~in $\Omega$ for
 all $t \in (0,T)$. Then passing to the limit as $|\tau| \to 0$, we
 conclude that $u^1 \leq u^2$ a.e.~in $Q_T$.
\end{proof}

\section{Long-time behavior of solutions}
\label{sec-ab} 

This section is devoted to proving Theorem \ref{T:asp}.
Let us begin with deriving a uniform estimate for $u(t)$ for $t \geq 0$.
 To do so, recall the construction of the unique solution $u = u(x,t)$ of
 \eqref{is1}--\eqref{ic} performed in the proof of Theorem \ref{T:ex}
 and particularly note that 
 \begin{equation}\label{recall}
  u_k \geq u_{k-1} \quad \mbox{ and } \quad g_k := \dfrac{u_k -
   u_{k-1}}{\tau_k} - \Delta u_k - f_k \geq 0
   \quad \mbox{ a.e.~in } \Omega.
 \end{equation}
It follows that
\begin{equation}\label{lo1}
h_k := g_k + f_k = \dfrac{u_k - u_{k-1}}{\tau_k} - \Delta u_k
\geq - \Delta u_k \quad \mbox{ a.e.~in } \Omega.
\end{equation}
We also recall the estimate \eqref{c-K2}, which gives
\begin{equation}\label{lo3}
f_k \leq h_k \leq (- \Delta u_{k-1}) \vee f_k
\quad \mbox{ a.e.~in } \Omega.
\end{equation}
Therefore by (H3) we find that
\begin{align*}
f_k \stackrel{\eqref{lo3}}\leq h_k 
 &\stackrel{\eqref{lo3}}{\leq} \left(-\Delta u_{k-1}\right) \vee f_k\\
 &\stackrel{\eqref{lo1}}{\leq} h_{k-1} \vee f^*\\
 &\leq (h_{k-2} \vee f^*) \vee f^*\\
 &= h_{k-2} \vee f^* 
 \leq \cdots \leq  h_1 \vee f^*
 \stackrel{\eqref{lo3}}\leq \left(-\Delta u_0\right) \vee f^*
\quad \mbox{ a.e.~in } \Omega,
\end{align*}
which together with the assumption that $u_0 \in X \cap V$ gives
\begin{align}
\|h_k\|_{L^2(\Omega)} 
&\leq \|\Delta u_0\|_{L^2(\Omega)} + \|f^*\|_{L^2(\Omega)} +
 \|f_k\|_{L^2(\Omega)}\nonumber\\
&\leq \|\Delta u_0\|_{L^2(\Omega)} + \|f^*\|_{L^2(\Omega)} +
 \|f\|_{L^\infty(0,\infty;L^2(\Omega))}.\label{1}
\end{align}
Here we used the fact that $\|f_k\|_{L^2(\Omega)} \leq
\|f\|_{L^\infty(0,\infty;L^2(\Omega))}$ for all $k$.
Moreover, set 
$$
\bar h_\tau (t) := \partial_t u_\tau(t) - \Delta \bar u_\tau(t)
 = h_k \quad \mbox{ for } \ t \in (t_{k-1},t_k].
$$
Recalling the convergence of approximate solutions obtained in the proof
of Theorem \ref{T:ex}, we observe that
$$
\bar h_\tau \to \partial_t u - \Delta u =: h \quad \mbox{ weakly in }
 L^2(0,T;L^2(\Omega)), 
$$
On the other hand, since $\{\bar h_\tau\}$ is bounded in
$L^\infty(0,T;L^2(\Omega))$ by \eqref{1}, we assure, up to a
(non-relabeled) subsequence, that
$$
\bar h_\tau \to h \quad \mbox{ weakly star in }
L^\infty(0,T;L^2(\Omega))
$$
as $|\tau| \to 0$. Moreover, from the lower
 semicontinuity of the $L^\infty$-norm in the weak star topology, we
 have, by \eqref{1},
$$
 \left\| h \right\|_{L^\infty(0,T;L^2(\Omega))}
 \leq \liminf_{\tau \to 0}
 \left\| \bar h_\tau \right\|_{L^\infty(0,T;L^2(\Omega))}
  \leq \|\Delta u_0\|_{L^2(\Omega)} + \|f^*\|_{L^2(\Omega)}
+ \|f\|_{L^\infty(0,\infty;L^2(\Omega))} 
$$
for each $T > 0$. Since the bound is independent of $T > 0$,
 one can derive that 
\begin{equation}\label{e:g-i:i}
\|h \|_{L^\infty(0,\infty;L^2(\Omega))} \leq 
\|\Delta u_0\|_{L^2(\Omega)} + \|f^*\|_{L^2(\Omega)}
+ \|f\|_{L^\infty(0,\infty;L^2(\Omega))}. 
\end{equation}

Note that $(\xi, v)_{L^2(\Omega)} = 0$ for all $v \in L^2(\Omega)$ with
$v \geq 0$ a.e.~in $\Omega$ and $\xi \in \partial \I(v)$. Thus testing
\eqref{ee2} by $\partial_t u(t)$, we have
\begin{align*}
\|\partial_t u(t)\|_{L^2(\Omega)}^2 + \dfrac 1 2 \dfrac{\d}{\d t} \|\nabla u(t)\|_{L^2(\Omega)}^2
 &= (f(t), \partial_t u(t))_{L^2(\Omega)}\\
 &= (f(t) - f_\infty, \partial_t u(t))_{L^2(\Omega)} + \dfrac{\d}{\d t} (f_\infty, u(t))_{L^2(\Omega)}\\
 &\leq \dfrac 1 2 \|\partial_t u(t)\|_{L^2(\Omega)}^2 + \dfrac 1 2 \|f(t) - f_\infty\|_{L^2(\Omega)}^2
 + \dfrac{\d}{\d t} (f_\infty, u(t))_{L^2(\Omega)}.
\end{align*}
Define an energy functional $E$ on $V$ by
$$
E(v) := \dfrac 1 2 \|\nabla v\|_{L^2(\Omega)}^2 - (f_\infty, v)_{L^2(\Omega)} \quad \mbox{ for } \
 v \in V.
$$
Then one has
\begin{equation}\label{e-ineq}
\dfrac 1 2 \|\partial_t u(t)\|_{L^2(\Omega)}^2 + \dfrac{\d}{\d t} E(u(t)) 
\leq \dfrac 1 2 \|f(t) - f_\infty\|_{L^2(\Omega)}^2 \quad \mbox{ for a.e. } \ t
 \geq 0,
\end{equation}
which implies the non-increase of the function
$$
t \mapsto E(u(t)) - \dfrac 1 2 \int^t_0 \|f(\tau) -
f_\infty\|_{L^2(\Omega)}^2 \, \d \tau
\quad \mbox{ for }  \ t \geq 0.
$$
Moreover, by using the Poincar\'e inequality (due to (H1)), we have
\begin{equation}\label{2}
E(v) \geq \dfrac 1 4 \|\nabla v\|_{L^2(\Omega)}^2 - C \|f_\infty\|_{L^2(\Omega)}^2 \quad \mbox{ for all } \ v \in V.
\end{equation}
Thus integrating \eqref{e-ineq} over $(0,s)$ and using (H2) and
 \eqref{2} we obtain
 \begin{align}
 \int^\infty_0 \|\partial_t u(t)\|_{L^2(\Omega)}^2 \, \d t \leq C, \label{e:1}\\
 \sup_{t \geq 0} \|\nabla u(t)\|_{L^2(\Omega)} \leq C, \label{e:2}
\end{align}
which also yields
\begin{align}
 \sup_{t \geq 0} \|\Delta u(t)\|_{V'} \leq C. \label{e:3}
 \end{align}
Moreover, by virtue of \eqref{e:g-i:i}, we have 
\begin{equation}\label{e:4}
 \|g\|_{L^\infty(0,\infty;L^2(\Omega))} 
  \leq  \|h\|_{L^\infty(0,\infty;L^2(\Omega))} + 
  \|f\|_{L^\infty(0,\infty;L^2(\Omega))}
  \leq M
\end{equation}
for some constant $M$.
Here we used that $g(t) = h(t) - f(t)$.

Let $I \subset (0,\infty)$ be the set of all $t \geq 0$ for which \eqref{ee2}
holds true and $\|g(t)\|_{L^2(\Omega)}$ is bounded by $M$ as in \eqref{e:4}.
Then the set $(0,\infty) \setminus I$ has zero Lebesgue measure. Recalling by
\eqref{e:1} and (H2) that
$$
\int^\infty_0 \left( \| \partial_t u(t) \|_{L^2(\Omega)}^2 + \|f(t) -
f_\infty\|_{L^2(\Omega)}^2 \right) \, \d t < \infty,
$$
one can take a sequence $s_n \in [n,n+1] \cap I$ such that
\begin{alignat}{4}
\partial_t u(s_n) &\to 0 \quad &&\mbox{ strongly in } L^2(\Omega),\label{s1}\\
 f(s_n) &\to f_\infty \quad &&\mbox{ strongly in } L^2(\Omega) \label{s2}
\end{alignat}
as $n \to \infty$.

Moreover, by using the preceding uniform (in $t$) estimates and using
 the compact embedding $V \hookrightarrow L^2(\Omega)$, we deduce,
 up to a (non-relabeled) subsequence, that
\begin{alignat}{4}
 u(s_n) &\to \stasol \quad &&\mbox{ weakly in } V,\\
 & &&\mbox{ strongly in } L^2(\Omega),\label{c:u-2}\\
 - \Delta u(s_n) &\to - \Delta \stasol \quad &&\mbox{ weakly in } V',\\
 - g(t) &\to \xi \quad &&\mbox{ weakly in }
 L^2(\Omega) 
\end{alignat}
with some $\stasol \in V$ and $\xi \in L^2(\Omega)$.
From the demiclosedness of $\partial \I$ in $L^2(\Omega)$ and the fact
that $-g(t) \in \partial \I(\partial_t u(t))$ for a.e.~$t \in
(0,\infty)$, it follows that $\xi \in \partial \I (0)$, that is, $\xi
\leq 0$ a.e.~in $\Omega$ by $\partial \I (0) = (-\infty,0]$. 
Moreover, by \eqref{ee2} and \eqref{s2}, we
 get $\xi - \Delta \stasol = f_\infty$, which leads us to $f_\infty +
 \Delta \stasol = \xi \leq 0$ a.e.~in $\Omega$. 
Furthermore, by \eqref{A1condition} along with Proposition
\ref{prop-GammaD}, the limit $\stasol$ belongs to $X$, since
$\stasol - \Delta \stasol = \stasol - \xi + f_\infty \in L^2(\Omega)$.

Therefore we derive that
\begin{align*}
 \|\nabla u(s_n)\|_{L^2(\Omega)}^2
 &= \left( - \Delta u(s_n), u(s_n) \right)_{L^2(\Omega)}\\
 &= \left( - \partial_t u(s_n), u(s_n) \right)_{L^2(\Omega)}
 + \left( g(s_n), u(s_n) \right)_{L^2(\Omega)}
 + \left( f(s_n), u(s_n) \right)_{L^2(\Omega)}\\
 &\to \left( - \xi + f_\infty, \stasol \right)_{L^2(\Omega)}
 = \left(- \Delta \stasol , \stasol \right)_{L^2(\Omega)} = \|\nabla \stasol\|_{L^2(\Omega)}^2.
\end{align*}
From the uniform convexity of $V$, it holds that
\begin{equation}\label{c:u-V}
u(s_n) \to \stasol \quad \mbox{ strongly in } V.
\end{equation}

We shall next verify the convergence of the solution $u(t)$ to the same limit
 $\stasol$ as $t \to \infty$, that is, $u(t_n) \to \stasol$ for any sequence $t_n
 \to \infty$ and the limit $\stasol$ is independent of the choice of the
 sequence $(t_n)$. Subtracting the stationary equation
$$
\partial \I (0) - \Delta \stasol \ni f_\infty
$$
from the evolution equation \eqref{ee2}, we see that 
$$
\partial_t u(t) + \partial \I(\partial_t u(t)) - \partial \I(0)
- \Delta \left( u(t) - \stasol \right) \ni f(t) - f_\infty.
$$
Test it by $\partial_t u(t)$ to get
$$
\dfrac 1 2 \|\partial_t u(t)\|_{L^2(\Omega)}^2 
+ \dfrac 1 2 \dfrac{\d}{\d t} \|\nabla (u(t) - \stasol)\|_{L^2(\Omega)}^2 
\leq \dfrac 1 2 \|f(t) - f_\infty\|_{L^2(\Omega)}^2.
$$
Integrate both sides over $(s_n, \tau)$ for $\tau > s_n$. Then it follows
 from \eqref{c:u-V} that
\begin{align*}
\dfrac 1 2 \sup_{\tau \geq s_n} \|\nabla (u(\tau) - \stasol)\|_{L^2(\Omega)}^2
&\leq \dfrac 1 2 \|\nabla (u(s_n) - \stasol)\|_{L^2(\Omega)}^2
+ \dfrac 1 2 \int^\infty_{s_n} \|f(t) - f_\infty\|_{L^2(\Omega)}^2 \,\d t\\
&\stackrel{\text{(H2)}}{\to} 0.
\end{align*}
Thus $u(t)$ converges to the limit $\stasol$ strongly in
 $V$ as $t \to \infty$. This completes the proof of the first half of
 the assertion.

We next prove the second half of the assertion.
In addition, assume that $f(x,t) \leq f_\infty(x)$ for a.e.~$x \in Q$
and let $\bar \stasol \in X \cap V$ be the unique solution of the variational
 inequality (VI)($u_0,f_\infty$). Then by Proposition \ref{th1} and
 Theorem \ref{th2} for $A = A_\sigma$ with $\sigma = 0$, $\bar \stasol$
 satisfies $-\Delta \bar \stasol \geq f_\infty$ a.e.~in $\Omega$, and
 moreover, we deduce that $U(x,t) := \bar \stasol(x)$ becomes a strong
 solution of \eqref{is1} by observing that
$$
\partial_t U \equiv 0 \ \mbox{ and } \
\Delta U(x,t) + f(x,t) \leq \Delta \bar \stasol(x) + f_\infty(x) \leq 0
\ \mbox{ a.e.~in } Q.
$$
Hence by the comparison principle for the evolutionary problem \eqref{is1}
(see Theorem \ref{T:comp}), we assure that $u(x,t) \leq \bar \stasol(x)$
for a.e.~$(x,t) \in Q$. Letting $t \to \infty$ and recalling
\eqref{c:u-2}, we obtain
$$
\stasol(x) \leq \bar \stasol(x) \quad \mbox{ for a.e. } x \in \Omega.
$$

On the other hand, since $\stasol$ belongs to $X \cap V$ and satisfies $\stasol
\geq u_0$ and $-\Delta \stasol \geq f_\infty$ in $V'$, applying the comparison
theorem for variational inequalities of obstacle type (see Lemma
\ref{lem-cp-vi}) to ${\rm (VI)}(u_0,f_\infty)$, we assure that
$\bar \stasol \leq \stasol$ a.e.~in $\Omega$. Consequently, we
 conclude that $\stasol = \bar \stasol$ a.e.~in $\Omega$.
 Thus we have proved the second half of the assertion of Theorem
 \ref{T:asp}. \qed

\section{Other equivalent formulations}\label{od}

In this section, we discuss other formulations of solutions for
\eqref{is1}--\eqref{ic} equivalent to those defined by Definition
\ref{def-ss}. Let us start with a \emph{complementarity form} of strong
solutions.

\begin{Th}\label{th-ws}
Let $f\in L^2(Q_T)$ and $u_0\in L^2(\Omega)$.
Then $u$ is a strong solution of the problem \eqref{is1}--\eqref{ic} on
 $[0,T]$, if and only if the following six conditions are satisfied\/{\rm :}
\begin{enumerate}[{\rm (V1)}]
\item
$u\in W^{1,2}(0,T;L^2(\Omega)) \cap
L^2(0,T;H^2(\Omega)\cap H^1(\Omega))$,
\item
$\partial_t u \ge 0$
a.e.~in $Q_T$,
\item
$\partial_t u -\Delta u-f
\ge 0$ a.e.~in $Q_T$,
\item
$\left(\partial_t u -\Delta u-f\right)
 \partial_t u =0$ a.e.~in $Q_T$,
\item
$\partial_\nu u = 0$ \ $\cH^{n-1}$-a.e.~on $\GammaN$ and $u = 0$ \
     $\cH^{n-1}$-a.e.~on $\GammaD$ \ for a.e.~$t \in (0,T)$,
\item
$u(0,\cdot)=u_0$.
\end{enumerate}
\end{Th}

\begin{proof}
If $u$ satisfies (V1) (or (i) of Definition \ref{def-ss}), one can define
 the following measurable subsets of $Q_T$:
\begin{align*}
Q_0&:=\{(x,t)\in Q_T \colon \partial_t u\neq (\Delta u+f)_+\},\\
Q_1&:=\{(x,t)\in Q_T \colon \partial_t u= (\Delta u+f)_+>0\},\\
Q_2&:=\{(x,t)\in Q_T \colon \partial_t u= (\Delta u+f)_+=0\},
\end{align*}
which are disjoint and satisfy $Q_T=Q_0\cup Q_1\cup Q_2$.

Let $u$ satisfy (i)--(iii) of Definition \ref{def-ss}. 
Conditions (V1), (V5) and (V6) follow immediately. 
From (ii) of Definition \ref{def-ss}, it follows that
$$
\cH^{n+1}(Q_0)=0,
$$
and moreover, by definition,
\begin{align*}
&\partial_t u>0,\quad \partial_t u-\Delta u-f= 0\quad\mbox{ a.e.~in }Q_1,\\
&\partial_t u=0,\quad \partial_t u-\Delta u-f\ge 0\quad\mbox{ a.e.~in }Q_2.
\end{align*}
Hence (V2), (V3) and (V4) follows.
Consequently, every strong solution $u$ of \eqref{is1}--\eqref{ic} in
 the sense of Definition \ref{def-ss} satisfies all the conditions
 (V1)--(V6).

Conversely, let $u$ satisfy (V1)--(V6). 
Conditions (i) and (iii) of Definition \ref{def-ss} follow from
 (V1), (V5) and (V6). 
 Let us next show that $\cH^{n+1}(Q_0)=0$, which is equivalent to the
 condition (ii) of Definition \ref{def-ss}. Define
\begin{align*}
Q^*:=\{(x,t)\in Q_T \colon \partial_t u \ge 0 \ \mbox{ and } \ \partial_t u-\Delta
 u-f\ge 0 \ \mbox{ at } (x,t)\}.
\end{align*}
Then it holds that
 $\cH^{n+1}(Q_T\setminus Q^*)=0$ by (V2) and (V3). 

We claim that
\begin{align}\label{QQ}
\partial_t u> 0,
\quad
\partial_t u -\Delta u -f> 0
\quad
\mbox{at each }\ (x,t) \in Q_0\cap Q^*.
\end{align}
Indeed, by the definitions of $Q_0$ and $Q^*$, 
$u$ satisfies the following conditions at each $(x,t) \in Q_0\cap Q^*$:
\begin{align}
&\partial_t u \neq (\Delta u+f)_+,\label{cc1}\\
&\partial_t u\ge 0,\label{cc2}\\
&\partial_t u -\Delta u -f\ge 0.\label{cc3}
\end{align}
If $\partial_t u=0$ at some $(x_0,t_0) \in Q_0\cap Q^*$, 
then $\Delta u+f>0$ at $(x_0,t_0)$ by \eqref{cc1}
and it contradicts \eqref{cc3}.
Hence, we obtain $\partial_t u>0$ at each point of $Q_0\cap Q^*$ by \eqref{cc2}.
Similarly, if $\partial_t u-\Delta u-f=0$ at some $(x_0,y_0) \in Q_0\cap Q^*$, 
then $0<\partial_t u=\Delta u+f=(\Delta u+f)_+$ at $(x_0,y_0)$, which
 contradicts \eqref{cc1}.
Thus, we obtain $\partial_t u-\Delta u-f>0$ in $Q_0\cap Q^*$ by \eqref{cc3}.

Since $Q_T=Q_0\cup Q_1\cup Q_2$ is a disjoint union and $\cH^{n-1}(Q_T
 \setminus Q^*) = 0$, we have
\begin{align*}
0&\stackrel{\text{(V4)}}=
\dint_{Q_T} \left( \partial_t u
-\Delta u-f\right)
\partial_t u
\,\d x \,\d t
=
\dint_{Q_0}\left( \partial_t u
-\Delta u-f\right)
\partial_t u
\,\d x \,\d t\\
&=
\dint_{Q_0\cap Q^*}\left(\partial_t u
-\Delta u-f\right)
\partial_t u 
\,\d x \,\d t.
\end{align*}
By \eqref{QQ}, we obtain $\cH^{n+1}(Q_0\cap Q^*)=0$;
otherwise the last integral is positive.
Hence we conclude that
\begin{align*}
\cH^{n+1}(Q_0)=\cH^{n+1}(Q_0\cap Q^*)+\cH^{n+1}(Q_0\setminus Q^*)=0.
\end{align*}
This completes the proof. 
\end{proof}

Finally, let us discuss a possible formulation of
\eqref{is1} in the sense of viscosity solutions. Set 
\begin{equation}\label{veq}
F(x,t,Y):= - \big( \mathrm{tr} \, Y +f(x,t) \big)_+
\quad
\mbox{ for } \ x\in\Omega, \ t\in (0,T), \ Y \in \R^{n\times n}_{\rm sym},
\end{equation}
where $\R^{n\times n}_{\rm sym}$ denotes the set of all symmetric $n
\times n$ real matrices. Then \eqref{eqn} is also written as
\begin{align*}
\partial_t u(x,t) + F(x,t,D^2u(x,t)) = 0,
\end{align*}
where $D^2u(x,t)\in \R^{n\times n}_{\rm sym}$ is the Hessian matrix 
of $u$. Since $F$ is degenerate elliptic, one may apply the theory of
viscosity solutions to prove the existence and uniqueness of viscosity
solutions of \eqref{veq} under suitable assumptions for $f(x,t)$
and the boundary condition. However, to the authors' knowledge, no
result on such a viscosity approach to \eqref{is1} has been obtained except
for~\cite{YamaN}. Moreover, the relation between the notion of viscosity
solutions and that of strong solutions for \eqref{is1} is widely open.
For further details of the theory of viscosity solutions,
we refer the reader to~\cite{UG},~\cite{BG},~\cite{Gig06} and references
therein.

\section*{Acknowledgments}

G.A.~is supported by JSPS KAKENHI Grant Number 25400163.
M.K.~is supported by JSPS KAKENHI Grant Number 26400195.
Both authors are also supported by the JSPS-CNR bilateral joint research
project: \emph{Innovative Variational Methods for Evolution Equations}.

\end{document}